\documentclass[11pt]{amsart}

\usepackage{epsf,amssymb,overpic,amscd,diagrams,psfrag,times}
\usepackage[all]{xy}

\usepackage{xr}
\newtheorem{thm}{Theorem}[section]
\newtheorem{prop}[thm]{Proposition}
\newtheorem{lemma}[thm]{Lemma}

\numberwithin{equation}{subsection}
\numberwithin{thm}{subsection}

\theoremstyle{definition}
\newtheorem{defn}[thm]{Definition}

\theoremstyle{remark}

\newtheorem{rmk}[thm]{Remark}

\DeclareMathAlphabet{\mathpzc}{OT1}{pzc}{m}{it}
\newcommand{\hh}{\mathpzc{h}}

\newcommand{\C}{\mathbb{C}}

\newcommand{\R}{\mathbb{R}}
\newcommand{\Z}{\mathbb{Z}}

\newcommand{\N}{\mathbb{N}}

\newcommand{\bdry}{\partial}
\newcommand{\s}{\vskip.1in}
\newcommand{\n}{\noindent}

\newcommand{\F}{\mathbb{F}}

\newcommand{\be}{\begin{enumerate}}
\newcommand{\ee}{\end{enumerate}}
\newcommand{\op}{\operatorname}
\newcommand{\bs}{\boldsymbol}

\begin{document}

\title[HF=ECH via open book decompositions]
{The equivalence of Heegaard Floer homology and embedded contact homology III: from hat to plus}

\author{Vincent Colin}
\address{Universit\'e de Nantes, 44322 Nantes, France}
\email{vincent.colin@univ-nantes.fr}

\author{Paolo Ghiggini}
\address{Universit\'e de Nantes, 44322 Nantes, France}
\email{paolo.ghiggini@univ-nantes.fr}
\urladdr{http://www.math.sciences.univ-nantes.fr/\char126 Ghiggini}

\author{Ko Honda}
\address{University of Southern California, Los Angeles, CA 90089}
\email{khonda@usc.edu} \urladdr{http://www-bcf.usc.edu/\char126 khonda}

\date{This version: August 7, 2012.}



\thanks{VC supported by the Institut Universitaire de France, ANR Symplexe, ANR Floer Power, and ERC Geodycon. PG supported by ANR Floer Power and ANR TCGD. KH supported by NSF Grants DMS-0805352 and DMS-1105432.}

\begin{abstract}
Given a closed oriented $3$-manifold $M$, we establish an isomorphism between the Heegaard Floer homology group $HF^+ (-M)$ and the embedded contact homology group $ECH(M)$.  Starting from an open book decomposition $(S,\hh)$ of $M$, we construct a chain map $\Phi^+$ from a Heegaard Floer chain complex associated to $(S,\hh)$ to an embedded contact homology chain complex for a contact form supported by $(S,\hh)$. The chain map $\Phi^+$ commutes up to homotopy with the $U$-maps defined on both sides and reduces to the quasi-isomorphism $\Phi$ from \cite{CGH2, CGH3} on subcomplexes defining the hat versions. Algebraic considerations then imply that the map $\Phi^+$ is a quasi-isomorphism.
\end{abstract}

\maketitle

\setcounter{tocdepth}{1}
\tableofcontents

\section{Introduction} \label{section: intro}

This is the last paper in the series which proves the isomorphism between certain Heegaard Floer homology and embedded contact homology groups. References from \cite{CGH2} (resp.\ \cite{CGH3}) will be written as ``Section I.$x$'' (resp.\ ``Section II.$x$'') to mean ``Section $x$'' of \cite{CGH2} (resp.\ \cite{CGH3}), for example.

Let $M$ be a closed oriented $3$-manifold.  Let $\widehat{HF}(M)$ and $HF^+(M)$ be the hat and plus versions of Heegaard Floer homology of $M$ and let $\widehat{ECH}(M)$ and $ECH(M)$ be hat and usual versions of the embedded contact homology of $M$. As usual, embedded contact homology will be abbreviated as ECH. In \cite{CGH1}, we introduced the ECH chain group $\widehat{ECC}(N,\bdry N)$ and showed that $\widehat{ECH}(N,\bdry N)\simeq \widehat{ECH}(M)$. In the papers \cite{CGH2,CGH3}, we defined a chain map
$$\Phi: \widehat{CF}(-M)\to \widehat{ECC}(N,\bdry N),$$
which induced an isomorphism
$$\Phi_*:\widehat{HF}(-M)\stackrel\sim\to \widehat{ECH}(M).$$

The goal of this paper is to extend the above result and prove the following theorem:

\begin{thm}\label{thm: HF+=ECH}
If $M$ is a closed oriented $3$-manifold, then there is a chain map $$\Phi^+:CF^+ (-M)\stackrel\sim\to ECC(M)$$
which is a quasi-isomorphism and which commutes with the $U$-maps up to homotopy.
\end{thm}

We use $\F=\Z /2\Z$ coefficients for both Heegaard Floer homology and ECH.

\begin{rmk}
The construction of $\Phi^+$ can be carried out with twisted coefficients as in Sections~I.\ref{P1-subsection: twisted coefficients phi map} and I.\ref{P1-subsection: defn of psi} of \cite{CGH2}.
\end{rmk}

Let $(S,\hh)$ be an open book decomposition for $M$, where $S$ is a genus $g\geq 2$ bordered surface with connected boundary and $\hh\in \op{Diff}(S,\bdry S)$.\footnote{The condition $g\geq 2$ is a technical condition which will used in the definition of $\Phi^+$.}  In particular we identify
$$M\simeq (S\times[0,1])/\sim,$$
where $(x,1)\sim (\hh(x),0)$ for all $x\in S$ and $(x,t)\sim (x,t')$ for all $x\in\bdry S$ and $t,t'\in[0,1]$. We write $S_t =S\times \{ t\}$ for $t\in[0,1]$. Let $\Sigma=S_0\cup -S_{1/2}$ be the Heegaard surface corresponding to $(S,\hh)$.

Given a pair $(\Sigma_0,\hh_0)$ consisting of a surface $\Sigma_0$ and $\hh_0\in \op{Diff}(\Sigma_0)$, we write the mapping torus of $(\Sigma_0,\hh_0)$ as:
$$N_{(\Sigma_0,\hh_0)}=(\Sigma_0\times[0,2])/(x,2)\sim (\hh_0(x),0).$$
The map $\Phi$, defined in Section~I.\ref{P1-subsection: defn of chain maps}, is induced by the cobordism $W_+$ which is an $S_0$-fibration and which restricts to a half-cylinder over $[0,1]\times S_0$ at the positive end and to a half-cylinder over the mapping torus $N_{(S_0,\hh)}$ at the negative end. We say that $W_+$ is a cobordism ``from $[0,1]\times S_0$\footnote{We will interchangeably write $[0,1]\times S_0$ and $S_0\times[0,1]$.  This is partly due to the fact that the open book is usually written as $(S\times[0,1])/\sim$ and the positive end of $W_+$ is a ``symplectization'' $\R\times[0,1]\times S_0$.} to $N_{(S_0,\hh)}$.''

The map $\Phi^+$ is induced by a cobordism $X_+$ from $[0,1]\times\Sigma$ to $M$ which extends $W_+$ and is described below.\footnote{The reader is warned that we are distinguishing $W_+$ from $X_+$, i.e., between lower and upper indices.} Although $\Phi$ was defined in terms of just one page $S_0$, we can no longer ignore the $S_{1/2}$ portion of $\Sigma$ when defining $\Phi^+$, since we do not know how to express $HF^+(-M)$ in terms of $S_0$.

A symplectic cobordism similar to $X_+$ is constructed by Wendl in \cite{We4}.

\subsubsection{The cobordism $X_+$} \label{subsubsection: brief sketch}

We give a description of $X_+=X_+^0\cup X_+^1\cup X_+^2$ and $W_+= W_+^0\cup W_+^1\cup W_+^2$ as topological spaces, where $W_+^i\subset X_+^i$ for $i=0,1,2$.  See Figure~\ref{figure: Wplus}. {\em The description given here is the simplified version of the actual construction, and the notation of Section~\ref{subsubsection: brief sketch} is not used outside of Section~\ref{subsubsection: brief sketch}.}

\begin{figure}[ht]
\begin{center}
\psfragscanon
\psfrag{A}{\small $S_{1/2}$}
\psfrag{B}{\small $S_0$}
\psfrag{C}{\small $B_+^0$}
\psfrag{D}{\small $D^2$}
\includegraphics[width=6cm]{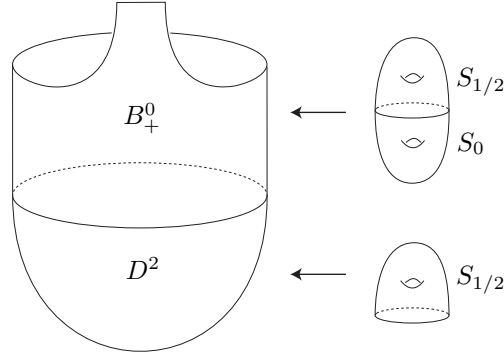}
\end{center}
\caption{Schematic diagram for $X_+^0\cup X_+^1$ which indicates the fibers over each subsurface.}
\label{figure: Wplus}
\end{figure}

First extend $\hh\in \op{Diff}(S_0,\bdry S_0)$ to $\hh^+\in \op{Diff}(\Sigma)$ so that $\hh^+|_{S_{1/2}}=id$. Let $N_{(\Sigma,\hh^+)}$ and $N_{(S_0,\hh)}$ be the mapping tori of $\hh^+$ and $\hh$ and let
$$\pi:[0,\infty)\times N_{(\Sigma,\hh^+)}\to [0,\infty)\times \R/2\Z$$ be the projection $(s,x,t)\mapsto (s,t)$. Then define $B^0_+ =([0,\infty ) \times \R /2\Z)-B_+^c$, where $B_+^c$ is the subset $[2,\infty)\times[1,2]$ with the corners rounded. We then set
$$X_+^0:= \pi^{-1}(B_+^0),\quad W_+^0:= \pi^{-1}(B_+^0)\cap ([0,\infty)\times N_{(S_0,\hh)}).$$

Next we set
$$X_+^1: = S_{1/2}\times D^2,\quad W_+^1:=\varnothing$$ and identify $\{0\}\times S_{1/2} \times \R/2\Z \subset \bdry X_+^0$ with $S_{1/2}\times \bdry D^2\subset \bdry X_+^1$ via the map $(0,x,t)\mapsto (x,e^{\pi it})$. Then one component of $\bdry (X_+^0\cup X_+^1)$ is given by $M=(\{0\}\times N_{(S_0,\hh)})\cup (\bdry S_0\times D^2)$.

Finally we set
$$X_+^2:=(-\infty,0]\times M,\quad W_+^2:= (-\infty,0]\times (\{0\}\times N_{(S_0,\hh)}),$$ where $\{0\}\times M$ is identified with $M$.

\subsubsection{Sketch of proof}
The proof of Theorem~\ref{thm: HF+=ECH} proceeds as follows:

\s\n
{\em Step 1.} Express the $U$-map on $HF^+ (-M)$ as a count of $I_{HF} =2$ curves that pass through a point, in analogy with the definition of $U$ in ECH. This is given by Theorem~\ref{thm: U-map}.

\s\n
{\em Step 2.} Construct a symplectic cobordism $(X_+ ,\Omega_{X_+} )$ from $[0,1] \times \Sigma$ to $M$, together with stable Hamiltonian and contact structures on $[0,1]\times \Sigma$ and $M$. This is the goal of Section~\ref{section: cobordism}.

\s\n
{\em Step 3.} Define the chain map $\Phi^+$ as a count of $I_{X_+} =0$ curves in $X_+$ and show that $\Phi^+$ commutes with the $U$-maps on both sides up to a chain homotopy $K$. This is done in Section~\ref{section: Phi^+}.

\s\n
{\em Step 4.} By an algebraic theorem (Theorem~\ref{thm: algebraic}), $\Phi^+$ is a quasi-isomorphism if a map
$$\Phi_{alg} :\widehat{CF} (-M)\to\widehat{ECC} (M),$$
defined using $\Phi^+$ and $K$, is a quasi-isomorphism.

\s\n
{\em Step 5.} By Theorem~\ref{thm: phi alg}, the map $\Phi_{alg}$ is a quasi-isomorphism. This is proved by relating $\Phi_{alg}$ to the quasi-isomorphism $\Phi$ from \cite{CGH2, CGH3}.

\section{Heegaard Floer chain complexes} \label{section: HF side}

The goal of this section is to introduce some notation and recall the definition of the chain complex $CF^+ (\Sigma,\bs\alpha,\bs\beta,z^f,J)$, whose homology is $HF^+(-M)$.

\subsection{Heegaard data}

Let $M$ be a closed oriented $3$-manifold and let $(S,\hh)$ be an open book decomposition for $M$.

We use the following notation, which is similar to that of Section~I.\ref{P1-subsubsection: Heegaard diagram compatible with S h}:
\begin{itemize}
\item $\Sigma =S_0\cup -S_{1/2}$ is the associated genus $2g$ Heegaard surface of $M$;
\item ${\bf a}=\{a_1 ,...,a_{2g}\}$ is a basis of arcs for $S$ and ${\bf b}$ is a small pushoff of ${\bf a}$ as given in Figure~I.\ref{P1-darkside};
\item $x_i$ and $x_i'$ are the endpoints of $a_i$ in $\partial S_0$ that correspond to the coordinates of the contact class and $x''_i$ is the unique point of $a_i \cap b_i \cap int (S_{1/2} )$;
\item $\bs{\alpha} = ({ \bf a} \times \{{1\over 2} \}) \cup ({\bf a} \times \{ 0\})$ and $\bs\beta = ({\bf b} \times \{{1\over 2}\}) \cup (\hh( {\bf a}) \times \{ 0\})$ are the collections of compressing curves on the Heegaard surface $\Sigma$;
\item $z^f$ is a point in the large (i.e., non-thin-strip) component of $S_{1/2} -\bs\alpha -\bs\beta$ and $(z')^f$ is a point which is close but not equal to $z^f$.
\end{itemize}
We say that the pointed Heegaard diagram $(\Sigma,\bs\alpha,\bs\beta,z^f)$ is {\em compatible with $(S,\hh)$.}

\begin{rmk}
{\em The orientation for $\Sigma$ is opposite to that of Section~I.\ref{P1-subsubsection: Heegaard diagram compatible with S h}.} This is done so that the triple $(S,{\bf a}, \hh({\bf a}))$, used in \cite{CGH2,CGH3}, embeds in $(\Sigma,\bs\alpha,\bs\beta)$ in an orientation-preserving manner.
\end{rmk}

\subsection{Symplectic data} \label{subsection: symplectic data}

The stable Hamiltonian structure on $[0,1] \times \Sigma$ with coordinates $(t,x)$ is given by $(\lambda,\omega)$, where $\lambda=dt$ and $\omega$ is an area form on $\Sigma$ which makes $(\bs\alpha,\bs\beta,z^f)$ {\em weakly admissible with respect to $\omega$}, i.e., each periodic domain has zero $\omega$-area. The plane field $\xi =\ker \lambda$ is equal to the tangent plane field of $\{ t\} \times \Sigma$ and the Hamiltonian vector field is $R=\frac{\partial}{\partial t}$.

We introduce the ``symplectization'' $$(X,\Omega)=(\R \times [0,1] \times \Sigma,ds\wedge dt +\omega),$$ where $(s,t)$ are coordinates on $\R\times[0,1]$. Let $\pi_B :X \rightarrow B=\R \times [0,1]$ be the projection along the fibers $\{(s,t)\}\times\Sigma$.

Let $J$ be an $\Omega_X$-admissible almost complex structure on $X$; we assume that $J$ is regular (cf.\ Lemma~I.\ref{P1-lemma: HF regularity} and \cite[Proposition~3.8]{Li}). We also define the Lagrangian submanifolds
$$L_{\bs\alpha}=\R\times\{1\}\times\bs\alpha,\quad L_{\bs\beta}=\R\times\{0\}\times\bs\beta.$$

\subsection{The chain complex $CF^+(\Sigma,\bs\alpha,\bs\beta,z^f,J)$}

In this subsection we recall the definition of the chain complex $CF^+(\Sigma,\bs\alpha,\bs\beta,z^f,J)$, whose homology group $$HF^+(\Sigma,\bs\alpha,\bs\beta,z^f,J)$$ is isomorphic to $HF^+(-M)$. This definition is due to Lipshitz~\cite{Li}, with one modification: we are using the ECH index $I_{HF}$ from Definition~I.\ref{P1-defn: ECH-type index for HF}. We will often suppress $J$ from the notation.

Let $\mathcal{S}=\mathcal{S}_{\bs\alpha,\bs\beta}$ be the set of $2g$-tuples ${\bf y} =\{y_1 ,\dots , y_{2g}\}$ of intersection points of $\bs\alpha$ and $\bs\beta$ for which there exists some permutation $\sigma \in \mathfrak{S}_{2g}$ such that $y_j \in \alpha_j \cap \beta_{\sigma (j)}$ for all $j$. Then $CF^+ (\Sigma,\bs\alpha,\bs\beta,z^f,J)$ is generated over $\F$ by pairs $[{\bf y} ,i]$, where ${\bf y}\in \mathcal{S}$ and $i\in \N$.

The differential $\bdry=\bdry_{HF}$ is given by
$$\bdry [{\bf y},i] = \sum_{[{\bf y}',j] \in {\mathcal S}\times \N} \langle \bdry [\mathbf{y},i],[\mathbf{y}',j]\rangle \cdot [{\bf y}',j],$$
where the coefficient $\langle \partial [{\bf y} ,i],[{\bf y}',j]\rangle$ is the count of index $I_{HF}=1$ finite energy holomorphic multisections in $(X,J)$ with Lagrangian boundary $L_{\bs\alpha}\cup L_{\bs\beta}$ from ${\bf y}$ to ${\bf y}'$, whose algebraic intersection with the holomorphic strip $\R \times [0,1] \times \{ (z')^f \}$ is $(i-j)$. We will often refer to such curves as curves {\em from $[{\bf y},i]$ to $[{\bf y}',j]$}.

Let us write $\bdry=\sum_{k=0}^\infty \bdry_k$, where $\bdry_k$ only counts curves whose algebraic intersection with $\R \times [0,1] \times \{(z')^f \}$ is $k$.

\section{The geometric $U$-map} \label{section: geometric U map}

\subsection{Introduction}

In \cite{OSz1,Li}, the $U$-map
$$U : CF^+ (\Sigma,\bs\alpha,\bs\beta,z^f)\rightarrow CF^+ (\Sigma,\bs\alpha,\bs\beta,z^f),$$
is defined algebraically as $U([{\bf y},i])=[{\bf y}, i-1]$. The goal of this section is to give a geometric definition of the $U$-map which is analogous to that of ECH.

Let $z^f, (z')^f$ be as before and let $z = (z^b,z^f )\in X=B \times \Sigma$, where $z^b\in int(B)$. Let $J^\Diamond$ be a generic $C^l$-small perturbation of $J$ such that $J^\Diamond=J$ away from a small neighborhood $N(z)\subset X$ of $z$ and such that $N(z)\cap (\R \times [0,1] \times \{(z')^f \})=\varnothing$. In particular, we assume that there are no $J^\Diamond$-holomorphic curves that are homologous to $\{pt\}\times \Sigma$ and pass through $z$.

\begin{rmk}
When we refer to ``$C^l$-close'' almost complex structures, etc., we assume that $l>0$ is sufficiently large.
\end{rmk}

\begin{defn} [Geometric $U$-map]
The {\em geometric $U$-map with respect to the point $z$} is the map:
$$U_z ([{\bf y},i]) = \sum_{[{\bf y'},j] \in {\mathcal S}\times \N} \langle U_z ([\mathbf{y},i]),[\mathbf{y'},j]\rangle \cdot [{\bf y'},j],$$
where the coefficient $\langle U_z ([{\bf y} ,i]), [{\bf y'},j] \rangle$ is the count of index $I_{HF} = 2$ finite energy holomorphic curves in $(X,J^\Diamond)$ with Lagrangian boundary $L_{\bs\alpha}\cup L_{\bs\beta}$ from $[{\bf y},i]$ to $[{\bf y}',j]$ that pass through $z$.
\end{defn}

By our choice of $J^\Diamond$, ``passing through $z$'' is a generic codimension $2$ condition, i.e., if $u$ is a simple curve from $[{\bf y},i]$ to $[{\bf y}',j]$ that passes through $z$, then $I(u)\geq 2$.  Moreover, by a simple ECH index count, an $I(u)\leq 3$ curve that passes through $z$ cannot have a fiber component.

\begin{prop} \label{prop: U chain map}
$U_z$ is a chain map.
\end{prop}

\begin{proof}
This follows from standard arguments in symplectic geometry by using analytical
results proved in \cite{Li}.
\end{proof}

\begin{thm}\label{thm: U-map}
There exists a chain homotopy
$$H:  CF^+ (\Sigma,\bs\alpha,\bs\beta,z^f)\rightarrow CF^+ (\Sigma,\bs\alpha,\bs\beta,z^f)$$
such that
\begin{equation} \label{eqn: Us}
U_z -U = H\circ \partial_{HF} + \partial_{HF} \circ H.
\end{equation}
Moreover, for all ${\bf y} \in \mathcal{S}$, one has $H([{\bf y} ,0])=0$.
\end{thm}

The rest of this section is devoted to the proof of Theorem~\ref{thm: U-map}.

\subsection{A model calculation}

Let $\Sigma$ be a closed surface of genus $k$. We consider the manifold $D\times \Sigma$, where $D=\{|z|\leq 1\}\subset \C$. Let $\pi_D: D \times \Sigma\to D$ and $\pi_{\Sigma}:D \times \Sigma\to \Sigma$ be the projections of $D \times \Sigma$ onto the first and second factors. Let $\bs\beta=\{\beta_1,\dots,\beta_k\}$ be the set of $\beta$-curves for $\Sigma$. Choose $z^f\in \Sigma- \bs\beta$ and let $z=(0,z^f)\in D \times \Sigma$.

Let $J=j_D \times j_{\Sigma}$ be a product complex structure on $D \times \Sigma$ and $J^\Diamond$ be a generic $C^l$-small perturbation of $J$ such that $J^\Diamond=J$ away from a small neighborhood of $z$. The key feature of $J^\Diamond$ is that all the $J^\Diamond$-holomorphic curves that pass through $z$ are regular.

We then define the moduli space $\mathcal{M}_{A}(D \times \Sigma,J^*)$, $*=\varnothing$ or $\Diamond$, of stable maps
$$u:(F,j)\to (D \times \Sigma,J^*)$$ in the class $A=[\{pt\}\times \Sigma]+k[D\times\{pt\}]\in H_2(D \times \Sigma, \partial D\times \bs\beta )$, such that $\bdry F$ has $k$ connected components and each component of $\bdry F$ maps to a distinct Lagrangian $\bdry D\times \beta_i$, $i=1,\dots,k$.  We choose points $w_i\in \beta_i$, $i=1,\dots,k$, and define
$${\bf w}=\{(1,w_1),\dots,(1,w_k)\}\subset  D \times \Sigma,$$
Then let
$$\mathcal{M}_{A}(D \times \Sigma, J^*;z,{\bf w})\subset \mathcal{M}_{A}(D \times \Sigma,J^*)$$
be the subset of curves that pass through $z$ and ${\bf w}$.  We use the modifier ``irr'' to denote the subset of irreducible curves.

\subsubsection{ECH index}

We briefly indicate the definition of the ECH index $I$ of a homology class $B\in H_2(D \times \Sigma,\bdry D \times\bs\beta)$ which admits a representative $F$ such that each component of $\bdry F$ maps to a distinct $\bdry D\times \beta_i$. Although we call $I$ the ``ECH index'', what we are defining here is a relative version of Taubes' index from \cite{T4}.

Let $\tau$ be a trivialization of $T\Sigma$ along $\bs\beta$, given by a nonsingular tangent vector field $Y_1$ along $\bs\beta$, and let $\tau'$ be a trivialization of $TD$ along $\bdry D$, given by an outward-pointing radial vector field $Y_2$ along $\bdry D$. Let $Q_{(\tau,\tau')}(B)$ be the intersection number between an embedded representative $u$ of $B$ and its pushoff, where the boundary of $u$ is pushed off in the direction given by $J(Y_1)$.

\begin{defn}
The {\em ECH index of the homology class $B$} is:
$$I(B)=c_1(T(D \times \Sigma)|_B,(\tau,\tau'))+ \mu_{(\tau,\tau')}(\bdry B)+Q_{(\tau,\tau')}(B).$$
\end{defn}

The following is the relative version of the adjunction inequality:

\begin{lemma}[Index inequality] \label{index ineq}
Let $u: (F,j)\to (D \times \Sigma,J^*)$ be a holomorphic curve in the class $B\in H_2(D \times \Sigma,\bdry D\times\bs\beta)$. Then
$$\op{ind}(u)+2\delta(u)= I(B),$$
where $\delta(u)\geq 0$ is an integer count of the singularities.
\end{lemma}

\begin{proof}
Similar to the proof of Theorem~I.\ref{P1-thm: index inequality for HF}.
\end{proof}

We now calculate some ECH and Fredholm indices:

\begin{lemma} \label{lemma: ECH index of A}
If $B=[\{pt\}\times \Sigma]+k_0[D\times\{pt\}]$ with $k_0\leq k$, then
$$I(B)= 2-2k+3k_0.$$
\end{lemma}

\begin{proof}
We compute that
\begin{align*}
I (B) & = I ([\{ pt\} \times \Sigma]+k_0[D \times \{ pt\}])\\
&= I ([\{ pt\} \times \Sigma ] )+ k_0\cdot I ([D \times \{ pt\}])+2k_0\cdot \langle[\{ pt\} \times \Sigma ],[D \times \{ pt\}]\rangle\\
&= (2-2k) + k_0\cdot 1 +2k_0  = 2-2k+3k_0.
\end{align*}
Here $\langle,\rangle$ denotes the algebraic intersection number.
\end{proof}

\begin{lemma}\label{lemma: Fredholm index of A}
If $B=[\{pt\}\times \Sigma]+k_0[D\times\{pt\}]$ with $k_0\leq k$ and $u$ is an irreducible $J^\Diamond$-holomorphic curve in the class $B$, then
$$\op{ind}(u)= 2-2k+3k_0-\delta(u).$$
\end{lemma}

\begin{proof}
Follows from Lemma~\ref{lemma: ECH index of A} and the index inequality.
\end{proof}

\subsubsection{Main result}
The following is the main result of this subsection:

\begin{thm} \label{thm: count is one}
If $J^\Diamond$ is generic, then the following hold:
\begin{enumerate}
\item $\mathcal{M}_{A}(D \times \Sigma,J^\Diamond;z,{\bf w})=\mathcal{M}_{A}^{irr}(D \times \Sigma,J^\Diamond;z,{\bf w})$;
\item $\mathcal{M}_{A}(D \times \Sigma,J^\Diamond;z,{\bf w})$ is compact, regular, and $0$-dimensional;
\item the curves of $\mathcal{M}_{A}(D \times \Sigma,J^\Diamond;z,{\bf w})$ are embedded; and
\item $\#\mathcal{M}_{A}(D \times \Sigma,J^\Diamond;z,{\bf w}) \equiv 1 \mbox{ mod } 2$.
\end{enumerate}
\end{thm}

Hence $\#\mathcal{M}_{A}(D \times \Sigma,J^\Diamond;z,{\bf w})$ is a certain relative Gromov-Witten invariant~\cite{IP1} which is computed to be $1\mbox{ mod } 2$.  (What we are really computing here is a relative Gromov-Taubes invariant~\cite{T4}, although the two invariants coincide in this case.)

\begin{proof}

(1) Let us write $\mathcal{M}=\mathcal{M}_{A}(D \times \Sigma,J^\Diamond;z,{\bf w})$. Arguing by contradiction, suppose $u\in \mathcal{M}-\mathcal{M}^{irr}$.  Then $u$ consists of an irreducible component $u_0$ which passes through $z$ and $k_0<k$ points of ${\bf w}$, together with $k-k_0$ copies of $D\times\{pt\}$. By Lemma~\ref{lemma: Fredholm index of A}, $\op{ind}(u_0)\leq 2-2k+3k_0$.  On the other hand, the point constraints are $(k_0+2)$-dimensional. Hence $u_0$ does not exist for generic $J^\Diamond$, which is a contradiction.

(2),(3) The compactness follows from the usual Gromov compactness theorem. The regularity of $\mathcal{M}$ is immediate from the genericity of $J^\Diamond$ and (1).  Lemma~\ref{lemma: Fredholm index of A} implies the dimension calculation, as well as (3).

(4) We degenerate the fiber $\Sigma$ into a union of $k$ tori which are successively attached to one another. We perform this pinching away from the curves $\bs\beta$ and make sure that each torus contains exactly one component of $\bs\beta$. Then by a degeneration/gluing argument as in Section~II.\ref{P2-subsubsection: reduction to torus for G 3},  it suffices to prove the proposition for $k=1$. The case $k=1$ is proved in Lemmas~\ref{lemma: regular moduli spaces} and \ref{lemma:  count is one for k=1}.
\end{proof}

\subsubsection{Doubling}

We now explain how to double $u\in \mathcal{M}_{A}(D \times \Sigma,J^\Diamond;z,{\bf w})$. For technical reasons we will assume that $(\Sigma,j_{\Sigma})$ admits an anti-holomorphic involution $\sigma_{\Sigma}$ so that the curve $\bs\beta$ is contained in the fixed point set of $\sigma_{\Sigma}$.

Let $\mathcal{D}(D \times \Sigma)=(D \times \Sigma)_1\cup (D \times \Sigma)_2$ be the double of $D\times\Sigma$, obtained by gluing two copies $(D \times \Sigma)_1$ and $(D \times \Sigma)_2$ of $D \times \Sigma$ along their boundaries $\bdry D\times \Sigma$ via the identification $(x,y)_1 \sim (x,\sigma_{\Sigma}(y))_2$, where the subscript $i=1,2$ indicates the $i$th copy. Let $\mathcal{S}$ be the involution of $\mathcal{D}(D \times \Sigma)$ given by
$$\mathcal{S}: (x,y)_1\mapsto (x,\sigma_{\Sigma}(y))_2, \quad \mathcal{S} : (x,y)_2\mapsto (x,\sigma_{\Sigma}(y))_1.$$  We then define the almost complex structure $\mathcal{D}(J^\Diamond)$ on $\mathcal{D}(D \times \Sigma)$ by taking $J^\Diamond$ on $(D \times \Sigma)_1$ and $\mathcal{S}(-J^\Diamond)$ on $(D \times \Sigma)_2$.

Given $u\in \mathcal{M}_{A}(D \times \Sigma,J^\Diamond;z,{\bf w})$, let $\mathcal{D}(u)$ be its double, obtained by gluing $u$ and $\mathcal{S}(u)$.  The map $\mathcal{D}(u)$ is holomorphic by the Schwarz reflection principle. Therefore it is an element of
$$\mathcal{M}_{\mathcal{D},J^\Diamond}^{irr}:=\mathcal{M}^{irr}_{\mathcal{D}(A)}(\mathcal{D}(D \times \Sigma),\mathcal{D}(J^\Diamond); z,\mathcal{S}(z),{\bf w}),$$
because it represents the class
$$\mathcal{D}(A):=2[\{ pt\} \times \Sigma ]+[S^2 \times \{ pt\}]$$ and passes through $3$ points:
$(1,w)$, $z$, and $\mathcal{S} (z)$, where $z\in (D \times \Sigma)_1$.
Conversely, if $v:(F,j)\to \mathcal{D}(D \times \Sigma)$ is in $\mathcal{M}_{\mathcal{D},J^\Diamond}$, then $\mathcal{S}\circ v:(F,-j)\to \mathcal{D}(D \times \Sigma)$ is also in $\mathcal{M}_{\mathcal{D},J^\Diamond}$. Thus all the curves of $\mathcal{M}_{\mathcal{D},J^\Diamond}$ come in pairs, except those that are $\mathcal{S}$-invariant, and the $\mathcal{S}$-invariant curves are those obtained by the doubling procedure.

Summarizing, we have:

\begin{lemma}\label{caltech}
$\#\mathcal{M}^{irr}_{A}(D \times \Sigma,J^\Diamond;z,{\bf w})\equiv\#\mathcal{M}^{irr}_{\mathcal{D},J^\Diamond} \mbox{ mod } 2.$
\end{lemma}

\subsubsection{The $k=1$ case}

For the rest of the subsection we assume $k=1$. We first consider the case where $J=j_D \times j_{\Sigma}$ is a product complex structure.

\begin{lemma} \label{degenerate curve}
If $k=1$, then:
\begin{enumerate}
\item $\mathcal{M}_{A}(D \times \Sigma,J;z,{\bf w})$ is a one-element set consisting of a degenerate curve $(D\times \{ w \})\cup (\{0\} \times \Sigma)$; and
\item $\mathcal{M}_{\mathcal{D},J}$ is a one-element set consisting of a degenerate curve $(S^2 \times \{ w \} )\cup (\{0\} \times \Sigma )_1\cup (\{0\} \times \Sigma )_2.$
\end{enumerate}
\end{lemma}

\begin{proof}
(1) follows from the homological constraint
$$A=[\{pt\}\times \Sigma]+[D\times\{pt\}].$$ If $u:(F,j)\to (D \times \Sigma,J)$ is a stable map in $\mathcal{M}_{A}(D \times \Sigma,J;z,{\bf w})$, then $\pi_D\circ u$ and $\pi_{\Sigma}\circ u$ are degree $1$ maps.  This implies that $F$ consists of two components $F_1, F_2$ and $\pi_D\circ u|_{F_1}$ and $\pi_{\Sigma}\circ u|_{F_2}$ are biholomorphisms. On the other hand, $\pi_\Sigma\circ u|_{F_1}$ maps to a point since $F_1$ is a disk and $\pi_D\circ u|_{F_2}$ maps to a point since otherwise the cardinality of $(\pi_D\circ u)^{-1}(pt)$ for generic $pt$ will be larger than $\deg(\pi_D\circ u)=1$. (2) is similar.
\end{proof}

Let $J^\Diamond$ be an almost complex structure which is $C^l$-close to $J$.  By Gromov compactness and Lemma~\ref{degenerate curve}, all the curves of $\mathcal{M}_{A}(D \times \Sigma,J^\Diamond;z,{\bf w})$ and $\mathcal{M}_{\mathcal{D},J^\Diamond}$ are close to the degenerate curves described in Lemma~\ref{degenerate curve}.

\begin{lemma} \label{lemma: regular moduli spaces}
If $k=1$ and $J^\Diamond$, ${\bf w}$, and $\bs\beta$ are generic, then the following hold:
\begin{enumerate}
\item $\mathcal{M}_{\mathcal{D},J^\Diamond}=\mathcal{M}^{irr}_{\mathcal{D},J^\Diamond}$;
\item the curves of $\mathcal{M}_{\mathcal{D},J^\Diamond}$ are embedded; and
\item $\mathcal{M}_{\mathcal{D},J^\Diamond}$ is compact, regular, and $0$-dimensional.
\end{enumerate}
\end{lemma}

\begin{proof}
(1) Arguing by contradiction, suppose $v\in \mathcal{M}_{\mathcal{D},J^\Diamond}- \mathcal{M}^{irr}_{\mathcal{D},J^\Diamond}$. If $v$ has a component $S^2\times\{pt\}$, then $v$ must also have a fiber component $\widetilde{v}$ which is close to $(\{0\}\times \Sigma)_i$, $i=1$ or $2$, since $v$ is close to a degenerate curve from Lemma~\ref{degenerate curve}(2).  On the other hand, if $v$ does not have a component $S^2\times\{pt\}$, then $v$ must have a fiber component $\widetilde{v}$ which is close to some $(\{0\}\times \Sigma)_i$. In either case, such a fiber component $\widetilde{v}$ cannot exist by the genericity of $J^\Diamond$. This proves (1).

(2) Let $v\in \mathcal{M}^{irr}_{\mathcal{D},J^\Diamond}$. The proof is similar to that of Lemma~\ref{thm: count is one}(3) and follows from the adjunction inequality \cite{M1,M2} (compare with Lemma~\ref{index ineq}): If
$$I(v)=c_1(v^*T\mathcal{D}(D \times \Sigma))+ Q(v),$$
where $Q(v)$ is the self-intersection number of $v$, then
$$\op{ind}(v) +2\delta(v)= I(v),$$
where $\delta(v)\geq 0$ is an integer count of the singularities.  Since $c_1(v^*T\mathcal{D}(D \times \Sigma))=2$ and $Q(v)=4$, it follows that $I(v)=6$. On the other hand,
$$\op{ind}(v)= -\chi(F)+2c_1(v^*T\mathcal{D}(D \times \Sigma))=-(-2)+2(2)=6,$$
where $F$ is the domain of $v$ with $\chi(F)=-2$.  Hence $v$ is embedded by the adjunction inequality.

(3) Since $v$ is embedded by (2) and $c_1(v^*T\mathcal{D}(D \times \Sigma))=2$, the regularity of $v$ without the point constraints follows from automatic transversality (cf.\ Hofer-Lizan-Sikorav~\cite[Theorem~1]{HLS}). The regularity with point constraints is the consequence of the genericity of $J^\Diamond$, ${\bf w}$, and $\bs\beta$. The rest of the assertion is immediate.
\end{proof}

\begin{lemma} \label{lemma: count is one for k=1}
If $k=1$ and $J^\Diamond$, ${\bf w}$, and $\bs\beta$ are generic, then $\#\mathcal{M}_{A}(D \times \Sigma,J^\Diamond;z,{\bf w})\equiv 1 \mbox{ mod } 2.$
\end{lemma}

\begin{proof}
By Lemma~\ref{caltech}, Lemma~\ref{lemma: regular moduli spaces}(1), and Theorem~\ref{thm: count is one}(1),
\begin{equation} \label{caltech prime}
\#\mathcal{M}_{A}(D \times \Sigma,J^\Diamond;z,{\bf w})\equiv\#\mathcal{M}_{\mathcal{D},J^\Diamond} \mbox{ mod } 2.
\end{equation}
We reduce the calculation of $\mathcal{M}_{\mathcal{D},J^\Diamond}$ to a calculation in McDuff-Salamon \cite[Example~8.6.12]{MS} by degenerating the base $S^2$ into two spheres along the curve corresponding to the boundary of the two copies of $D$.

More precisely,
let $$B=[S^2\times\{pt\}]+[\{pt\}\times \Sigma]\in H_2(S^2\times \Sigma),$$
let $I$ be a product complex structure on $S^2\times \Sigma$, and let $I^\Diamond$ be a $C^l$-small perturbation of $I$ such that $I=I^\Diamond$ away from a neighborhood of $(0,z^f)$. Here we are viewing $S^2\simeq \C\cup\{\infty\}$. Let $$\mathcal{M}_{B}:=\mathcal{M}_{B}(S^2\times \Sigma,I^\Diamond; (0,z^f),(\infty,w))$$ 
be the moduli spaces of $I^\Diamond$-curves in the class $B$ that pass through $(0,z^f)$ and $(\infty,w)$, and let $\mathcal{M}'_{B}$ be the subset of curves in $\mathcal{M}_{B}$ that are contained in a neighborhood of $(S^2 \times \{ w \} )\cup(\{0\} \times \Sigma )$. By McDuff-Salamon~\cite[Example~8.6.12]{MS}, $\#\mathcal{M}'_{B}=1$. By the symplectic sum theorem of Ionel-Parker~\cite[p.\ 940]{IP2} (also see Section~II.\ref{P2-subsubsection: reduction to torus for G 3}),  $$\#\mathcal{M}_{\mathcal{D},J^\Diamond}\equiv (\#\mathcal{M}'_{B})^2\equiv 1 \mbox{ mod } 2.$$
Hence $\#\mathcal{M}_{A}(D \times \Sigma,J^\Diamond;z,{\bf w})\equiv 1 \mbox{ mod } 2$ by Equation~\eqref{caltech prime}.
\end{proof}

\subsection{Family of cobordisms}

We now describe a family of marked points $z_\tau\in X$ and a family of almost complex structures $J_{\tau}^\Diamond$ on $X$ for $\tau\in[0,1)$, as well as their limits for $\tau=1$. These families give rise to the chain homotopy $H$ of Theorem~\ref{thm: U-map}.

Let $z^b_\tau \in int (B)$, $\tau\in[0,1)$, be a family of points such that $z^b_0 =z^b$, $\lim_{\tau\to 1} z^b_\tau=(0,0)$, and $z^b_\tau\in \{s=0\}$ for $\tau\in[{1\over 2},1)$. Then let $z_\tau=(z^b_\tau,z^f)\in X$.

Assume that the almost complex structure $J$ on $X$ is a product complex structure on $\R\times[0,\varepsilon]\times\Sigma$ for $\varepsilon>0$ small.  We then define a family of $C^l$-small perturbations $J_{\tau}^\Diamond$, $\tau\in[0,1)$, of $J$ such that $J_{\tau}^\Diamond=J$ away from a small neighborhood $N(z_\tau)$ of $z_\tau$ and $$N(z_\tau)\cap (\R \times [0,1] \times \{(z')^f \})=\varnothing.$$

In the limit $\tau=1$, the base $\widetilde{B}$ is $(B\sqcup D)/\sim$, where $D=\{|z|\leq 1\}\subset \C$ and $\sim$ identifies $(0,0)\in B$ with $-1\in D$, and the total space $\widetilde{X}$ is $(X\sqcup (D\times\Sigma))/\sim$, where $((0,0),x)\sim (-1,x)$ for all $x\in \Sigma$. See Figure~\ref{figure: bubbling}. We write $w^b$ for the node $[(0,0)]=[-1]\in \widetilde{B}$.  Let $\pi_{B}: X \to B$ and $\pi_{D}:D\times\Sigma\to D$ be the projections onto the first factors.

\begin{figure}[ht]
\begin{center}
\psfragscanon
\psfrag{a}{\tiny $\tau\to 1$}
\psfrag{b}{\tiny $z^b_\tau$}
\psfrag{c}{\tiny $z^b_1$}
\psfrag{d}{\tiny $w^b$}
\includegraphics[width=5cm]{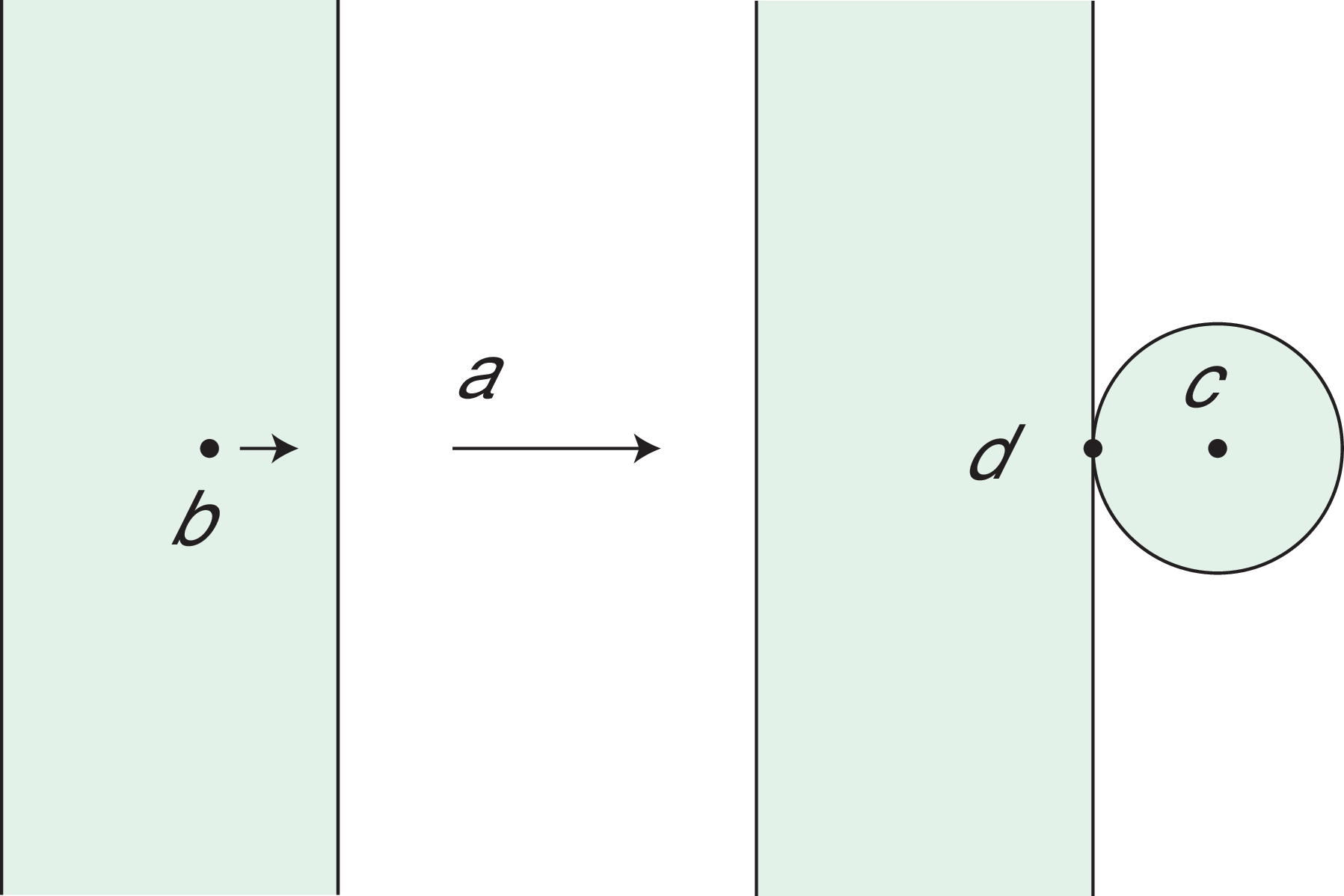}
\end{center}
\caption{}
\label{figure: bubbling}
\end{figure}

The limit $z_1$ of $z_\tau$ is in $D\times\Sigma$ and we assume that $z_1^b=0\in int(D)$. When $\tau=1$, the almost complex structure  $J_1^\Diamond$ restricts to the complex structure $J$ on $X$ 
and to the almost complex structure $ J^{D, \Diamond}$, where
$J^D$ is a product complex structure on $D\times\Sigma$ and $J^{D, \Diamond}$ is a $C^l$-small perturbation of $J^{1,D}$ such that $J^{D, \Diamond}=J^D$ away from a small neighborhood $N(z_1)$ of $z_1$ and $$N(z_1)\cap (D \times \{(z')^f \})=\varnothing.$$

The Lagrangian boundary condition for $\tau\in[0,1)$ is $L_{\bs\alpha}\cup L_{\bs\beta}$. In the limit $\tau=1$, we use $L_{\bs\alpha}\cup L_{\bs\beta}$ for $X$ and $\bdry D\times \bs\beta$ for $D\times \Sigma$.

\subsection{Proof of Theorem~\ref{thm: U-map}}
Let $u_{\tau_i}$, $\tau_i\to 1$, be a sequence of $I_{HF}=2$ curves in $(X,J_{\tau_i}^\Diamond)$ from $[{\bf y} ,i]$ to $[{\bf y'},i-k]$ that pass through $z_{\tau_i}$. Applying Gromov compactness, we obtain the limit $\tilde{u}=u_B\cup u_D$, where $u_B\subset X$, $u_D\subset D\times \Sigma$, and $u_D$ passes through $z_1$. Components of $\tilde{u}$ that map to the fiber $\{w^b\}\times\Sigma$ will be viewed as components of $u_D$.

\begin{lemma} \label{lemma: A} $\mbox{}$
\begin{enumerate}
\item $[u_D]=k_0[\{ pt\} \times \Sigma ] +2g  [D \times \{ pt\}]\in H_2(D\times \Sigma)$ for some $0<k_0\leq k$.
\item $I(u_D)=2k_0+2g\geq 2g+2$.
\end{enumerate}
\end{lemma}

\begin{proof}
(1) $\deg(\pi_{D}\circ u_D)=2g$, since $u_{\tau_i}$ is a degree $2g$ multisection of $X$ for each $\tau_i$, away from a neighborhood of $z^b_{\tau_i}$. Also, since $\langle u_{\tau_i},B \times \{ (z^f)'\}\rangle=k$ for all $\tau_i$, it follows that $\langle u_D, D\times \{(z^f)'\}\rangle=k_0$, where $0<k_0\leq k$. Here $k_0>0$ since $u_D$ passes through $z_1$.

(2) is a consequence of (1) and computations as in the proof of Lemma~\ref{lemma: ECH index of A}. We remind the reader that the genus of $\Sigma$ is  $2g$.
\end{proof}

\begin{lemma}\label{lemma: B}
$I(u_D)=2g+2$ and $I_{HF}(u_B)=0$. In particular, ${\bf y}={\bf y}'$, $u_B$ consists of $2g$ trivial strips, and $k_0=k=1$.
\end{lemma}

\begin{proof}
The gluing constraints give $I_{HF} (u_\tau)=I (u_D )+I_{HF} (u_B )-2g=2$. By the regularity of $J$ and the index inequality, we have $I_{HF}(u_B)\geq 0$. The first sentence of the lemma then follows from Lemma~\ref{lemma: A}(2); the second sentence is a consequence of the first.
\end{proof}

The first sentence of Theorem~\ref{thm: U-map} follows from the usual construction of chain homotopies in Floer theory.
By Lemma~\ref{lemma: B}, $U_z$ is chain homotopic to $aU$, where $a$ is the count of holomorphic curves $u_D$ in $(D\times \Sigma, J^{D, \Diamond})$ that pass through $z_1$ and ${\bf w}=((1,y_1),\dots,(1,y_{2g}))$, where ${\bf y}=\{y_1,\dots,y_{2g}\}$. Since $a=1$ modulo $2$ by Theorem~\ref{thm: count is one}, $U_z$ is chain homotopic to $U$.

Next we prove the second sentence of Theorem~\ref{thm: U-map}.  For all ${\bf y} \in \mathcal{S}$, $H([{\bf y},0])$ is obtained by counting $I_{HF}=1$ curves that pass through $z_\tau$ for some $\tau \in (0,1)$ and that do not cross the holomorphic strip $\R \times [0,1] \times \{ (z')^f \}$. There are no such curves since  $\R \times [0,1] \times \{ z^f \}$ is holomorphic and homologous to $\R \times [0,1] \times \{ (z')^f \}$: if a curve passes through $z_\tau$, its intersection with $\R \times [0,1] \times \{ z^f \}$ is strictly positive by the positivity of intersections,
and so is its intersection with $\R \times [0,1] \times \{ (z')^f \}$.

\section{The cobordism $X_+$}\label{section: cobordism}

In this section we give the construction of the symplectic cobordism $(X_+,\Omega_{X_+})$ from $[0,1] \times \Sigma$ to $M$, together with the Lagrangian submanifold $L_{\bs\alpha} \subset \partial X_+$.

\subsection{Construction of $(X_+,\Omega_{X_+})$}\label{subsection: W+}

We describe the construction of $X_+$, leaving some key details for later\footnote{Compare with the description in Section~\ref{subsubsection: brief sketch}, keeping in mind that the notation will be slightly different.}: First we construct fibrations $\pi_0: X_+^0\to B^0_+$ and $\pi_1: X_+^1\to D^2$ with fibers diffeomorphic to $\Sigma$ and $S_{1/2}$.  Here $B^0_+ =([0,\infty ) \times \R /2\Z)-B_+^c$ with coordinates $(s,t)$ and $B_+^c$ is the subset $[2,\infty)\times[1,2]$ with the corners rounded. We then glue $X_+^0$ and $X_+^1$ and smooth a boundary component $\mathcal{B}$ of $X_+^0\cup X_+^1$ to obtain $\widetilde{\mathcal{B}}\simeq M$. Finally we attach the negative end $X_+^2=(-\infty,0]\times \widetilde{\mathcal{B}}$ to obtain $X_+$.

Let $\delta>0$ be a small irrational number and $N$ a large positive number which depends on $\delta$ and whose dependence will be described later.

\begin{lemma} \label{properties}
There exists a symplectic manifold $(X_+,\Omega_{X_+})$ which depends on $\delta>0$ and which satisfies the following:
\begin{enumerate}
\item There is a symplectic surface $S_{z^f}:=\{z^f\}\times(B^0_+\cup D^2)$, obtained by gluing sections $\{z^f\}\times B^0_+\subset X_+^0$ and $\{z^f\}\times D^2\subset X_+^1$.
\item $\Omega_{X_+}=d\Theta^+$ for some $1$-form $\Theta^+$ on $X_+-N(S_{z^f})$, where $N(S_{z^f})$ is a small neighborhood of $S_{z^f}$.
\item $\Theta^+$ is exact on the Lagrangian submanifold $L_{\bs\alpha}\subset \bdry X_+$.
\item On the positive end
$$\pi_0^{-1}([3,\infty)\times [0,1])=[3,\infty)\times\Sigma\times [0,1]\subset X_+^0$$
of $X_+$, $\Omega_{X_+}$ restricts to $\widetilde\omega+ds\wedge dt$, where $\widetilde\omega$ is an area form on $\Sigma$. Moreover,
$$L_{\bs \alpha}\cap \{s\geq 3\} = ([3,\infty)\times\{0\}\times \bs\beta')\cup ([3,\infty)\times\{1\}\times\bs\alpha),$$
where $\bs\beta'$ is isotopic to $\bs \beta$.
\item On the negative end $X_+^2$ of $X_+$, $\Omega_{X_+}$ restricts to the negative symplectization of a contact form $\lambda_-$ on $\widetilde{\mathcal{B}} \simeq M$ which is adapted to the open book decomposition $(S,\hh)$.
\item The manifold $\widetilde{\mathcal{B}}\simeq M$ admits a decomposition into three disjoint pieces: the mapping torus $N_{(S_0,\hh_0)}$, where $\hh_0$ is isotopic to $\hh$ relative to $\bdry S_0$, a closed neighborhood $N(K)$ of the binding $K$, and an open thickened torus $\mathcal{N}$ in between that we refer to as the ``no man's land''.
\item All the orbits of the Reeb vector field $R_{\lambda_-}$ of $\lambda_-$ in $int(N(K))\cup\mathcal{N}$ have $\lambda_-$-action $\geq {1\over 2\delta}-\kappa$, where $\kappa>0$ is independent of $\delta$. Moreover, $T_+=\partial N(K)$ (resp.\ $T_-=\bdry N_{(S_0,\hh_0)}$) is a positive (resp.\ negative) Morse-Bott torus of meridian orbits.
\item There is an embedding of $W_+$, defined in Section~I.\ref{P1-acorn}, into $X_+$ such that the restriction $\pi_1: W_+\cap X_+^0\to B^0_+$ is a fibration with fiber $S_0$, $W_+\cap X_+^1=\varnothing$, $W_+\cap X_+^2=[0,\infty)\times N_{(S_0,\hh_0)}$, and $W_+$ is disjoint from $N(S_{z^f})$.
\end{enumerate}
Here $X_+$, $\Omega_{X_+}$, $\Theta^+$, $L_{\bs\alpha}$, and $\lambda_-$ depend on $\delta>0$.
\end{lemma}

The $S^1$-family $\mathcal{P}_+$ (resp.\ $\mathcal{P}_-$) of simple orbits of $T_+$ (resp.\ $T_-$) can be viewed equivalently as a pair $e'$, $h'$ (resp.\ $e$, $h$) consisting of an elliptic orbit and a hyperbolic orbit. The proof of Lemma~\ref{properties} will be given in Section~\ref{subsection: proof of properties}.

Let $A_{[-1,N]}\simeq [-1,N] \times S^1$ be a small neighborhood of $\bdry S_0 =\{ 0\} \times S^1$ in $\Sigma$ with coordinates $(r_1,\theta_1)$, such that $z^f \notin A_{[-1,N]}$, $A_{[-1,0]}\subset S_0$ and $A_{[0,N]}\subset S_{1/2}$. Here we write $A_\mathfrak{I}=\mathfrak{I}\times S^1$ if $\mathfrak{I}$ is a subset of $[-1,N]$.  Also let $N(z^f) \subset S_{1/2} - A_{[0,N]}-\bs\alpha -\bs\beta$ be a small ball $D_\tau=\{r'\leq \tau\}$ about $z^f$, where we are using polar coordinates $(r',\theta')$.

The actual construction of $(X_+, \Omega_{X_+})$ is a bit involved, and consists of several steps.

\s\n {\bf Step 1.}
The following lemma is a rephrasing of Lemma~I.\ref{P1-lemma: change from h to h_0} and its proof.

\begin{lemma}
After possibly isotoping $\hh$ relative to $\bdry S_0$, there exists a factorization $\hh=\hh_0\circ \hh_1$ and a contact form $\lambda= f_t(x)dt+\beta_t(x)$, $(x,t)\in S_0\times [0,2]$, on $N_{(S_0,\hh_0)}$ with Reeb vector field $R_{\lambda}$, such that the following hold:
\begin{enumerate}
\item $\hh:S_0\times\{0\}\stackrel\sim\to S_0\times\{0\}$ is the first return map of $R_{\lambda}$.
\item $\hh$ has no elliptic periodic point of period $\leq 2g$ in $int(S_0 )$, as required for technical reasons
in II.\ref{P2-thm: isomorphism}.
\item $\hh_0=id$ on $A_{[-1/2,0]}$.
\item $\hh_1$ is the flow of $R_{\lambda}$ from $S_0\times\{0\}$ to $S_0\times\{2\}$.
\item $R_{\lambda}$ is parallel to $\bdry_t$ on $(S_0-A_{[-1,0]})\times [0,2]$. In particular, $\hh_1=id$ on $S_0-A_{[-1,0]}$.
\item $f_t(r_1,\theta_1)=1+\varepsilon r_1^2/2$ and $\beta_t(r_1,\theta_1)=(C+r_1)d\theta_1$ on $A_{[-1/2,0]}$, for $\varepsilon>0$ sufficiently small and $C>0$. In particular, $f_t$ and $\beta_t$ are independent of $t$ and $R_{\lambda}$ is parallel to $\bdry_t -\varepsilon r_1\bdry_\theta$ on $A_{[-1/2,0]}$.
\item $|d_2 f_t|_{A_{[-1/2,0]}}|_{C^0}\leq \delta$ and ${1\over 2}\leq f_t\leq 2$.
\end{enumerate}
Here $\varepsilon>0$ depends on $\delta>0$, $d_2$ is the differential in the $S_0$-direction, and the $C^0$-norm is with respect to a fixed Riemannian metric on $S_0$.
\end{lemma}

\s\n
{\bf Step 2.} We then extend $\hh_0,\hh_1,\hh\in \mbox{Diff}(S_0,\bdry S_0)$ to $\hh_0^+,\hh_1^+,\hh^+=\hh^+_0\circ \hh^+_1\in\mbox{Diff}(\Sigma)$ and the contact form $\lambda$ to the contact form  ${\lambda_+}=f_tdt+\beta_t$ to $N_{(\Sigma-N(z^f),\hh_0^+)}$, all of which depend on $\delta>0$, as follows:
\begin{itemize}
\item[(3')] $\hh_0^+=id$ on $S_{1/2}$.
\item[(4')] $\hh_1^+|_{\Sigma-N(z^f)}$ is the flow of $R_{\lambda_+}$ from $(\Sigma-N(z^f))\times\{0\}$ to $(\Sigma-N(z^f))\times\{2\}$ and $\hh_1^+|_{N(z^f)}=id$.
\item[(5')] $f_t$ and $\beta_t$ are independent of $t$ on $S_{1/2}-N(z^f)$. Hence $R_{\lambda_+}$ is parallel to $\bdry_t+ X_f$, where $X_f$ is the Hamiltonian vector field satisfying $i_{X_f}\omega= d_2f$ and $\omega$ is an area form on $\Sigma$ which agrees with $d_2\beta_t$ on $\Sigma-N(z^f)$.
\item[(6a')] $f_t(r',\theta')=const>0$ and $\beta_t(r',\theta')=(-C'+r')d\theta'$ near $\bdry N(z^f)$, for $-C'>0$. In particular, $R_{\lambda_+}$ is parallel to $\bdry_t$ near the mapping torus of $\bdry N(z^f)$.
\item[(6b')] $f_t(r_1,\theta_1)=1+\varepsilon r_1^2/2$ near $A_{\{0\}}$ and $\beta_t(r_1,\theta_1)=(C+r_1)d\theta_1$ on $A_{[0,N]}$.
\item[(7')] $|d_2 f_t|_{S_{1/2}-N(z^f)}|_{C^0}\leq \delta$ and ${1\over 2}\leq f_t|_{S_{1/2}-N(z^f)}\leq 2$.
\end{itemize}

Without loss of generality we may assume that $\bs\alpha\times\{1\}$ is Legendrian with respect to ${\lambda_+}$. This is an easy consequence of the Legendrian realization principle; see for example \cite[Theorem~3.7]{H}.

\s\n {\bf Step 3} (Construction of $(X_+^0,\Omega_{X_+}^0)$).
Let
$$\widetilde{X}_+^0=([0,\infty)\times\Sigma\times [0,2])/(s,x,2)\sim (s,\hh^+_0(x),0)$$ and let
$\pi_0: \widetilde{X}_+^0 \to [0,\infty)\times \R/2\Z$ be the projection $(s,x,t)\mapsto (s,t)$. We then set
$$X_+^0:= \pi_0^{-1}(B_+^0).$$

Let $g: [0,{1\over 2}]\to \R$ be a smooth function such that $g(r)=1+\varepsilon r^2/2$ near $r=0$, $0<g'(r)\leq \delta$ for $r\in(0,{1\over 2})$, $g'({1\over 2})=0$, and $g({1\over 2})=1+\varepsilon$. In particular, this requires $2\varepsilon<\delta$. Then let
$$\lambda_{+,s}=f_{s,t}dt+\beta_t,\quad s\in [0,\infty),$$ be a $1$-parameter family of contact forms\footnote{Note that $\beta_t$ does not depend on $s$.} on $N_{(\Sigma-N(z^f),\hh_0^+)}$ such that the following hold:
\begin{itemize}
\item[(a)] $\lambda_{+,s}={\lambda_+}$ if $s\geq {3\over 2}$ or $(x,t)\in N_{(S_0,\hh_0)}$.
\item[(b)] $\lambda_{+,s}$ is independent of $s$ if $s\in[0,{1\over 2}]$.
\item[(c)] $f_{0,t}(r_1,\theta_1)=g(r_1)$ on $A_{[0,1/2]}$.
\item[(d)] $f_{0,t}|_{S_{1/2}-A_{[0,1/2]}-N(z^f)}=1+\varepsilon$.  In particular, $d\lambda_{+,0}= d_2\beta_t$ and $R=\bdry_t$ on the mapping torus of $S_{1/2}-A_{[0,1/2]}-N(z^f)$.
\item[(e)] $f_{s,t}$ is a constant $C_s>0$ near $\bdry N(z^f)$.
\item[(f)] $|d_2 f_{s,t}|_{A_{[-1/2,0]}\cup S_{1/2}-N(z^f)}|_{C^0}\leq \delta$, $|\bdry_s f_{s,t}|_{A_{[-1/2,0]}\cup S_{1/2}-N(z^f)}|_{C^0}\leq \delta$  and ${1\over 2}\leq f_{s,t}|_{\Sigma-N(z^f)}\leq 2$ for all $s,t$.
\end{itemize}

We then define:
$$\Omega_{X_+}^0:= \widetilde\omega +ds\wedge dt,$$
where
$$\widetilde\omega=\left\{
\begin{array}{cl}
d\lambda_{+,s} & \mbox{on }~ X_+^0 - (N(z^f)\times B^0_+);\\
\omega & \mbox{on }~ N(z^f)\times B^0_+;
\end{array}
\right.$$
and $\omega$ is an area form on $\Sigma$ which agrees with $d_2\beta_t$ on $\Sigma-N(z^f)$. The $2$-form $\Omega_{X_+}^0$ is symplectic by an easy calculation which uses (f).

\s\n {\bf Step 4} (Construction of $(X^1_+,\Omega_{X_+}^1)$ and primitives $\Theta_0^+,\Theta_1^+$).
Let
$$X^1_+: =S'_{1/2} \times D^2, \quad S'_{1/2}: =S_{1/2}-A_{[0,1/2]}.$$
We use polar coordinates $(r_2,\theta_2)$ on $D^2=\{r_2\leq 1\}$. We identify neighborhoods of $\{0\}\times S'_{1/2} \times \R/2\Z \subset \bdry X_+^0$ and $S'_{1/2}\times \bdry D^2\subset \bdry X_+^1$ as follows:
\begin{align*}
\phi_{01}:[-\varepsilon',\varepsilon']\times S'_{1/2}\times \R/2\Z & \stackrel\sim\to S'_{1/2}\times \{(r_2,\theta_2)~|~ e^{-\pi \varepsilon'}\leq r_2\leq e^{\pi\varepsilon'}\},\\
(s,x,t) & \mapsto (x,e^{\pi s},\pi t),
\end{align*}
where $\varepsilon'>0$ is sufficiently small.

Let $\omega_{D^2}$ be an area form on $D^2$ satisfying:
$$\omega_{D^2}= \left\{  \begin{array}{cl}
r_2dr_2d\theta_2 & \mbox{near }~ r_2=0;\\
{1\over \pi^2 r_2} dr_2d\theta_2 & \mbox{near }~ r_2=1.
\end{array}\right.$$
We then define
$$\Omega_{X_+}^1: = \widetilde\omega|_{S'_{1/2}} + \omega_{D^2}.$$
An easy calculation shows that $\omega_{D^2}=ds\wedge dt$, and hence $\Omega_{X_+}^1=\Omega_{X_+}^0$, on their overlap.

We write $\omega_{D^2}=d(\phi(r_2) d\theta_2)$, where $\phi: [0,1]\to \R$ satisfies
$$\phi(r_2)= \left\{ \begin{array} {cl}
r_2^2/ 2 & \mbox{near }~ r_2=0;\\
{1\over \pi^2}\log r_2 + {1\over 10} & \mbox{near }~ r_2=1.
\end{array}\right.$$
Then $\phi(r_2)d\theta_2=(s+{\pi\over 10}) dt$ on their overlap. The choice of the constant ${\pi\over 10}<1$ will be used
in the proof of Lemma~\ref{lemma: energy bound}. We then define primitives $\Theta_i^+$ of $\Omega_{X_+}^i$, $i=0,1$, as follows:
\begin{align}
\label{primitive 1}
\Theta_0^+ & = (s+\pi/10)dt +\lambda_{+,s} ~\mbox{ on }~  X^0_+-(N(z^f)\times B^0_+);\\ \label{primitive 2}
\Theta_1^+ & = \phi(r_2) d\theta_2 +\lambda_{+,0} ~\mbox{ on }~ X^1_+ -(N(z^f)\times D^2).
\end{align}
We have $\Theta_0^+=\Theta_1^+$ on their overlap.

\s\n {\bf Step 5} (Corner smoothing).
We now have a $4$-manifold $X^0_+ \cup X^1_+$ with a concave corner along $(\partial S'_{1/2} ) \times \partial D^2$. The component $\mathcal{B}$ of $\partial (X_+^0\cup X_+^1)$ that contains the corner is homeomorphic to $M$ and $(\partial S'_{1/2} )\times D^2$ is a neighborhood of the binding $(\partial S'_{1/2} )\times \{ 0\}$.

In this step we round the corner of $\mathcal{B}$ to obtain the smoothing $\widetilde{\mathcal{B}}\subset X_+^0\cup X_+^1$.  We write $\widetilde{\mathcal{B}}_i= \widetilde{\mathcal{B}}\cap X_+^i$, $i=0,1$. We define the contact form $\lambda_-$ on $\widetilde{\mathcal{B}}$ so that  $\lambda_-|_{\widetilde{\mathcal{B}}_i}=\Theta_i^+|_{\widetilde{\mathcal{B}}_i}$, $i=0,1$. Here the notation $|_A$ refers to the pullback to $A$. See Figure~\ref{figure: rounding}.

\begin{figure}[ht]
\begin{center}
\psfragscanon
\psfrag{A}{\small $W_0^+$}
\psfrag{B}{\small $W_1^+$}
\psfrag{C}{\small $\mathcal{B}$}
\psfrag{D}{\small $\widetilde{\mathcal{B}}_1$}
\psfrag{E}{\small $\widetilde{\mathcal{B}}_{01}$}
\psfrag{F}{\small $\widetilde{\mathcal{B}}_{00}$}
\psfrag{a}{\small $r_1$}
\psfrag{b}{\small $r_2$}
\psfrag{c}{\small $s$}
\psfrag{1}{\small $0$}
\psfrag{2}{\small $\varepsilon_1$}
\psfrag{3}{\small ${1\over 2}-\varepsilon_0$}
\psfrag{4}{\small ${1\over 2}$}
\psfrag{5}{\small ${1\over 2}+\varepsilon_0$}
\psfrag{6}{\small $k_0$}
\psfrag{7}{\small $1$}
\includegraphics[width=8cm]{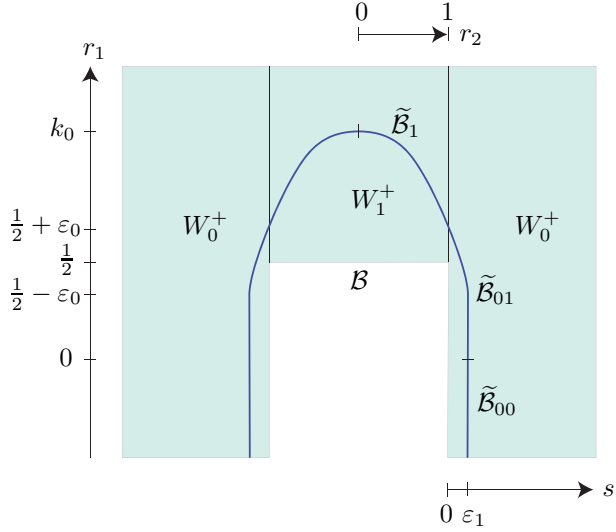}
\end{center}
\caption{Schematic diagram for rounding the corner of $\mathcal{B}$. The diagram shows a neighborhood $N(\mathcal{B})$ of $\mathcal{B}$, where we are projecting $X^0_+\cap N(\mathcal{B})$ to coordinates $(s,r_1)$ and $X^1_+\cap N(\mathcal{B})$ to coordinates $(r_0,r_1)$. }
\label{figure: rounding}
\end{figure}

\s\n
{\em Construction of $\widetilde{\mathcal{B}}_0$.}
There exist $\varepsilon_0,\varepsilon_1>0$ small with ${\varepsilon_1\over 2 \varepsilon_0}<\delta$ and $\varepsilon_1< \varepsilon'$ and a smooth map $\psi: [0,{1\over 2}+\varepsilon_0]\to \R$ such that:
\begin{itemize}
\item $\psi(r_1)=\varepsilon_1$ on $[0,{1\over 2}-\varepsilon_0]$;
\item $-\delta\leq \psi'(r_1)\leq 0$ on $[{1\over 2}-\varepsilon_0, {1\over 2}+\varepsilon_0]$; and
\item $\psi({1\over 2}+\varepsilon_0)=0$ and $\psi'({1\over 2}+\varepsilon_0)=-\delta$.
\end{itemize}
We then let $\widetilde{\mathcal{B}}_0= \widetilde{\mathcal{B}}_{00}\cup \widetilde{\mathcal{B}}_{01}$, where
\begin{align*}
\widetilde{\mathcal{B}}_{00}& = \{s=\varepsilon_1\}\times N_{(S_0,\hh_0)}\\
\widetilde{\mathcal{B}}_{01}& = \{ s=\psi(r_1), r_1\in[0,1/2+\varepsilon_0]\}\times \R/2\Z\times S^1.
\end{align*}
Here $\R/2\Z\times S^1$ has coordinates $(t,\theta_1)$.

\begin{lemma} \label{description1}
There exists $r_1^*\in (0,{1\over 2}+\varepsilon_0)$ such that each orbit in $ \widetilde{\mathcal{B}}_{01}\cap \{r_1\not=r_1^*\}$ is directed by some $\bdry_t+ \delta' \bdry_{\theta_1}$, where $0<\delta'\leq \delta$ or $-\delta\leq \delta'<0$, and $\widetilde{\mathcal{B}}_{01}\cap \{r_1=1+\varepsilon_0\}$ is directed by $\bdry_t+\delta\bdry_{\theta_1}$.
\end{lemma}

\begin{proof}
The $1$-form $\lambda_-|_{\widetilde{\mathcal{B}}_{00}}$ is clearly a contact form and
\begin{equation}\label{lambda 01}
\lambda_-|_{\widetilde{\mathcal{B}}_{01}}=(\psi(r_1)+ f_{0,t}(r_1,\theta_1)+\pi/ 10) dt +(C+r_1)d\theta_1,
\end{equation}
with respect to coordinates $(r_1,\theta_1,t)$. The Reeb vector field $R_{\lambda_-}$ is parallel to $\bdry_t - {\bdry\over \bdry r_1}(\psi+f_{0,t})\bdry_{\theta_1}$.
Let $r_1^*\in[0,{1\over 2}+\varepsilon_0]$ be the point where ${\bdry \over \bdry r_1}(\psi+ f_{0,t})=0$. Then $0< -{\bdry\over \bdry r_1}(\psi+f_{0,t})\leq \delta$ for $r_1\in [r_1^*,{1\over 2}+\varepsilon_0]$, $-{\bdry\over \bdry r_1}(\psi+f_{0,t})({1\over 2}+\varepsilon_0)=\delta$, and $0<{\bdry\over \bdry r_1}(\psi+f_{0,t})\leq \delta$ for $r_1\in (0,r_1^*)$, which imply the lemma.
\end{proof}

\s\n
{\em Construction of $\widetilde{\mathcal{B}}_1$.} Let $\zeta:[0,1]\to \R$ be a smooth map such that:
\begin{itemize}
\item $\zeta(r_2)=k_0-k_1 r_2^2/2$ near $r_2=0$, where $k_0,k_1\gg 0$;
\item $\zeta''<0$ on $(0,1]$;
\item $\zeta(1)= {1\over 2}+\varepsilon_0$.
\end{itemize}
We then define $\widetilde{\mathcal{B}}_1= \{ r_1= \zeta(r_2)\}$.

\begin{lemma}\label{description2}
There exist $k_0,k_1\gg 0$, $N=N(k_0,k_1)\gg 0$, and $\zeta$ such that $R_{\lambda_-}|_{\widetilde{\mathcal{B}}_1}$ is directed by $\pi\bdry_{\theta_2}+\delta \bdry_{\theta_1}$.
\end{lemma}

\begin{proof}
$\lambda_-|_{\widetilde{\mathcal{B}}_1}$ is given by
\begin{equation}\label{lambda 1}
\lambda_-|_{\widetilde{\mathcal{B}}_1}=(\phi(r_2)+(1+\varepsilon)/\pi)d\theta_2 + (C+\zeta(r_2))d\theta_1,
\end{equation}
with respect to coordinates $(\theta_1,r_2,\theta_2)$.  The Reeb vector field $R_{\lambda_-}$ is parallel to $\pi\bdry_{\theta_2} -{\phi'\over \zeta'}\bdry_{\theta_1}$. By choosing $k_0,k_1\gg 0$, $N(k_0,k_1)\gg 0$, and $\zeta$ suitably, we may assume that $-{\phi'\over \zeta'}(r_2)=\delta$ for all $r_2\in(0,1]$.
\end{proof}

We also define $N(K)\subset \widetilde{\mathcal{B}}$ as the closed neighborhood of the binding $K=\{r_2=0\}$ that is bounded by the torus $\{r_1=r_1^*\}$. The region $\mathcal{N}=\{0<r_1< r_1^*\}\subset \widetilde{\mathcal{B}}$ will be called ``no man's land''.

\s\n {\bf Step 6} (Construction of $(X_+^2,\Omega_{X_+}^2)$).
Let $X_+^{01}\subset X_+^0\cup X_+^1$ be the closure of the component of $(X_+^0\cup X_+^1)-\widetilde{\mathcal{B}}$ that does not contain $\mathcal{B}$.  We then glue the negative cylindrical end
$$(X_+^2,\Omega_{X_+}^2): =((-\infty ,0] \times \widetilde{\mathcal{B}},d(e^s\lambda_-))$$ to $X^{01}_+$ along $\widetilde{\mathcal{B}}$, where $s$ is the coordinate for $(-\infty,0]$. This concludes the construction of $(X_+,\Omega_{X_+})$.

\subsection{Further definitions} \label{subsection: Lagrangian} $\mbox{}$

\s\n{\em Hamiltonian structure on $\Sigma\times[0,1]$.}
Let $\overline\omega=\widetilde\omega|_{s={3\over 2}}$. Then the Hamiltonian structure on $\Sigma\times[0,1]$ at the positive end of $X_+$ is given by $(dt,\overline\omega|_{\Sigma\times[0,1]})$.  Let $\hh^+_2$ be the flow of the corresponding Hamiltonian vector field from $\Sigma\times\{0\}$ to $\Sigma\times\{1\}$. Note that we do not necessarily have $\hh^+_2=id$ by construction.

\s\n{\em Lagrangian submanifold $L_{\bs\alpha}$.} As in Section~I.\ref{P1-subsubsection: Lagrangian}, we define the Lagrangian submanifold $L_{\bs\alpha}\subset \partial X_+$ by placing a copy of $\bs\alpha$ on the fiber $\pi^{-1}(3,1)$ over $(3,1)\in \partial B_+^0$ and using the symplectic connection $\Omega_{X_+}$ to parallel transport $\bs\alpha$ along the boundary component $(\partial B_+ ^0)\cap \{ s\geq 1\}$ of $B_+^0$. Observe that
\begin{equation} \label{eqn: L}
L_{\bs\alpha}\cap \{s\geq 3\}= ([3,\infty)\times \{0\}\times \hh^+\circ (\hh^+_2)^{-1}(\bs\alpha))\cup ([3,\infty)\times\{1\}\times \bs\alpha).
\end{equation}

\begin{lemma} \label{isotopic}
$\bs\beta':=\hh^+\circ (\hh^+_2)^{-1}(\bs\alpha)$ is isotopic to $\bs\beta$.
\end{lemma}

\begin{proof}
First observe that $\hh^+_1$ and $\hh^+_2$ are isotopic to the identity. Then $\hh^+$ is isotopic to $\hh^+_0$ where  $\hh^+_0|_{S_{1/2}}=id$ and $\hh^+_0|_{S_0}$ is isotopic to $\hh$. The lemma then follows.
\end{proof}

\s\n{\em Submanifolds $S_z$, $C_\theta$, and $\mathcal{H}$.}
Given $z\in N(z^f)$, let
$$S_{z} =\{z\}\times (B^0_+\cup D^2),$$ where $\{z\}\times B^0_+\subset X_+^0$ and $\{z\} \times D^2\subset X_+^1$. Also let
$$C_\theta =(\{\theta\} \times B_+^0) \cup (\{\theta\} \times (-\infty ,0]\times \R/2\Z ),$$
where $\theta\in \bdry S_0$, and let $\mathcal{H} =\cup_{\theta\in \partial S_0} C_\theta$.

\s\n{\em Definition of $W_+$.}
Let $W_+$ be the closure of the component of $X_+-\mathcal{H}$ which is disjoint from $S_{(z')^f}$. In particular, the restriction $\pi_1: W_+\cap X_+^0\to B^0_+$ is a fibration with fiber $S_0$, $W_+\cap X_+^1=\varnothing$, and $W_+\cap X_+^2=[0,\infty)\times N_{(S_0,\hh_0)}$. The cobordism $W_+$ is diffeomorphic to the cobordism used to define the map $\Phi$ in Section~I.\ref{P1-subsection: symplectic cobordisms}.

\subsection{Proof of Lemma~\ref{properties}} \label{subsection: proof of properties}

(1), (5), (6), (8) are clear from the construction.

(2) follows by letting $\Theta^+=\Theta^+_i$, $i=0,1,2$, where defined.

(3) By construction, $L_{\bs\alpha}$ is Lagrangian and $d\Theta^+|_{L_{\bs\alpha}}=0$. It then suffices to observe that $\Theta^+=0$ on $L_{\bs\alpha}\cap \pi^{-1}(3,1)$. This follows from the fact that $\bs\alpha\times\{1\}$ is a Legendrian submanifold of $(N_{(\Sigma-N(z^f),\hh_0^+)},\lambda_+)$.

(4) The first sentence follows from the construction and the second sentence follows from Lemma~\ref{isotopic}.

(7) By Lemma~\ref{description2}, the Reeb vector field $R_{\lambda_-}$ has no closed orbits in $\widetilde{\mathcal{B}}_1$ since $\delta>0$ is irrational. By Lemma~\ref{description1} and Equation~\eqref{lambda 01}, each orbit of $R_{\lambda_-}$ in $ \widetilde{\mathcal{B}}_{01}\cap \{r_1\not=r_1^*\}$ has $\lambda_-$-action $\geq {1\over 2\delta}-(C+{1\over 2})$, where $C>0$ is independent of $\delta$. The second sentence of (7) is immediate from the construction of $\lambda_-$.

\section{The chain map $\Phi^+$} \label{section: Phi^+}

The goal of this section is to define the chain map
$$\Phi^+: CF^+(\Sigma,\bs\alpha,\bs\beta,z^f)\to ECC(M,\lambda_-),$$
which is induced by the symplectic cobordism $(X_+,\Omega_{X_+})$ and an admissible almost complex structure $J^+$.  We take $\bs\beta=\hh^+_2\circ\hh^+\circ (\hh^+_2)^{-1}(\bs\alpha)$, in view of Equation~\eqref{eqn: L} and Lemma~\ref{isotopic} and the fact that $\hh^+_2$ is the flow of the  Hamiltonian vector field of $\widetilde\omega|_{s=s_0}$, $s_0\gg 0$, from $\Sigma\times\{0\}$ to $\Sigma\times\{1\}$.

{\em For simplicity we identify $X_+\cap \{s\geq s_0\}\simeq [s_0,\infty)\times[0,1]\times\Sigma$ with coordinates $(s,t,x)$ so that $\hh_2^+=id$ and the Hamiltonian vector field is $\bdry_t$.}

\subsection{Almost complex structures}

Let $\overline\omega=\widetilde\omega|_{s=3/2}$.

\begin{lemma}
There exists a sequence $(\overline{\lambda}_\tau,\overline\omega)$, $\tau\in[0,1]$, of stable Hamiltonian structures on $N_{(S_0,\hh_0)}$ such that $\overline{\lambda}_1=\lambda$, $\overline{\lambda}_\tau$ is a contact form for $\tau>0$, and $\overline{\lambda}_0=dt$. The $1$-forms $\overline{\lambda}_\tau=f_{t,\tau}dt+\beta_{t,\tau}$ can be normalized so that ${1\over 2}< |f_{t,\tau}|\leq 2$.
\end{lemma}

\begin{proof}
Follows from the discussion of Section I.\ref{P1-subsection: interpolation}.
\end{proof}

\begin{defn}
An almost complex structure $J^+$ on $X_+$ is {\it $(X_+,\Omega_{X_+})$-admissible} if the following hold:
\begin{enumerate}
\item $J^+$ is tamed by $\Omega_{X_+}$;
\item $J^+$ is $s$-invariant for $\{s\geq {3\over 2}\}\cap X_+^0$ and is adapted to the stable Hamiltonian structure $(dt,\overline\omega|_{\Sigma\times[0,1]})$ at the positive end;
\item $J^+$ is $s$-invariant for $\{s\leq -{1\over 2}\}\cap X_+^2$ and is adapted to the contact form $\lambda_-$ at the negative end;
\item the restriction $J_+$ of $J^+$ to $W_+$ is $C^l$-close to a regular admissible almost complex structure $J_+^0$ on $W_+$ with respect to $(\overline\lambda_0,\overline\omega)$ (cf.\ Definitions~I.\ref{P1-defn: admissible J for W plus} and I.\ref{P1-regular pear});
\item the surfaces $S_{(z')^f}$ and $C_\theta$ are $J^+$-holomorphic for all $\theta\in \partial S_0$.
\end{enumerate}
Let $J,J'$ be the adapted almost complex structures that agree with $J^+$ at the positive and negative ends.
\end{defn}

Note that (4) imposes additional conditions on $\Omega_{X_+}$ and $\lambda_-$. In practice, the order in which we construct $\Omega_{X_+}$ and $J^+$ is a little convoluted:
(i) choose a regular $J_+^0$, (ii) choose $\tau>0$ sufficiently small and $J_+$ sufficiently close to $J_+^0$, (iii) construct $\Omega_{X_+}$ using $\overline\lambda_\tau$ in place of $\lambda$, and (iv) extend $J^+$ to the rest of $X_+$.

Let $\mathcal{J}_{X_+}$ be the set of all $(X_+,\Omega_{X_+})$-admissible almost complex structures.

\subsection{The ECH index}

Let $\mathcal{P}=\mathcal{P}_{\lambda_-}$ be the set of simple orbits of $R_{\lambda_-}$ and let $\mathcal{O}=\mathcal{O}_{\lambda_-}$ be the set of orbit sets constructed from $\mathcal{P}$.

Let $J^+ \in \mathcal{J}_{X_+}$ be an admissible almost complex structure. Let $\mathcal{M}_{J^+}({\bf y},\gamma)$ be the set of holomorphic maps $u:(F,j)\to (X_+,J^+)$ from ${\bf y}\in \mathcal{S}_{\bs\alpha,\bs\beta}$ to $\gamma\in \mathcal{O}$, such that $\bdry F$ is mapped to a distinct component of $L_{\bs\alpha}$ and each component is used exactly once. Elements of $\mathcal{M}_{J^+}({\bf y},\gamma)$ will be called {\em $X_+$-curves.}

Let $\check{X}_+$ be $X_+$ with the ends $\{s>3\}$ and $\{s<-1\}$ removed and let
$$Z_{{\bf y},\gamma}=(L_{\bs\alpha}\cap \check{X}_+)\cup (\{3\}\times [0,1]\times{\bf y}) \cup (\{-1\}\times \gamma)$$
as in Section~I.\ref{P1-subsubsection: W plus curves}. The class $[u]$ of $u\in \mathcal{M}_{J^+}({\bf y},\gamma)$ is the relative homology class of the compactification $\check u$ in $H_2 (\check{X}_+ ,Z_{{\bf y}, \gamma})$.  Given $A \in H_2 (\check{X}_+ ,Z_{{\bf y}, \gamma})$, we write $\mathcal{M}_{J^+} ({\bf y},\gamma,A)\subset \mathcal{M}_{J^+} ({\bf y},\gamma)$ for the subset of $X_+$-curves $u$ in the class $A$.

\begin{defn}[Filtration $\mathcal{F}$]
Given a $X_+$-curve $u$, we define
$$\mathcal{F}(u)=\langle [u],S_{(z')^f}\rangle,$$
where $\langle,\rangle$ is the algebraic intersection number. Since $S_{(z')^f}$ is a holomorphic divisor, $\mathcal{F}(u)\geq 0$.
\end{defn}

The definition of the ECH index given in Section~I.\ref{P1-subsection: ECH index W plus minus} also extends directly to our case.  The ECH index of a $X_+$-curve from ${\bf y}$ to $\gamma$ in the class $A$ is denoted by $I_{X_+}(\gamma,A)$.

\subsection{Homology of $X_+$}

The goal of this subsection is to compute $H_2(X_+)$.  We introduce some notation which will be used only in this subsection: $N=N_{(S_0, \hh_0)}$, $N_0=N_{(S_{1/2}, \hh^+_0|_{S_{1/2}})}$ and $\overline{N}=N_{(\Sigma, \hh^+_0)}$.

\begin{lemma} \label{weakness}
$H_2(N) \cong H_2(M)$ and $H_1(N) \cong H_1(M) \oplus \Z$, where the extra $\Z$ factor is generated by a meridian of the binding.
\end{lemma}

\begin{proof}
The lemma follows from the exact sequence of the pair $(M,N)$.
\end{proof}

\begin{lemma} \label{cough}
$H_2(X_+^0) \cong H_2(N) \oplus H_2(N_0) \oplus H_2(\Sigma)$.
\end{lemma}

\begin{proof}
Observe that $X_+^0$ is homotopy equivalent to $\overline{N}$. We compute $H_2(\overline{N})$ using the Mayer-Vietoris sequence:
$$H_2(N \cap N_0) \stackrel{i}\to H_2(N) \oplus H_2(N_0) \to H_2(\overline{N}) \to H_1(N \cap N_0) \stackrel{j}\to H_1(N) \oplus H_1(N_0).$$
Since $i=0$ and $\ker j=\Z\langle\partial S_0\rangle = \Z\langle\partial S_{1/2}\rangle$, the lemma follows.
\end{proof}

\begin{lemma} \label{flu}
$H_2(X_+) \cong H_2(M)  \oplus H_2(\Sigma)$.
\end{lemma}

\begin{proof}
$X_+$ is homotopy equivalent to $X_+^0 \cup X_+^1$ and $X_+^0 \cap X_+^1 \cong N_0$. Since $X_+^1$ is homotopy equivalent to $S_{1/2}$,  the Mayer-Vietoris sequence becomes:
$$H_2(N_0) \stackrel{i}\to H_2(X_+^0) \to H_2(X_+) \to H_1(N_0) \stackrel{j}\to H_1(X_+^0) \oplus H_1(S_{1/2}).$$
The map $i$ surjects onto the factor $H_2(N_0)$ in the decomposition of $H_2(X_+^0)$ coming from Lemma~\ref{cough}. The map $j$ is injective, since
$H_1(N_0) \cong H_1(S_{1/2}) \oplus H_1(S^1)$
by the K\"unneth formula, the restriction $j: H_1(S_{1/2})\to H_1(S_{1/2})$ is an isomorphism, and the restriction $j:H_1(S^1)\to H_1(\overline{N})\simeq H_1(X_+^0)$ is injective because the image of the generator of $H_1(S^1)$ is dual to the fiber $\Sigma$. The lemma then follows from Lemma~\ref{weakness}.
\end{proof}

\subsection{Energy bound}

\begin{defn}
Let $\mathcal{C}_+$ be the set of nondecreasing functions $[0,+\infty )\rightarrow [0,1]$ and let $\mathcal{C}_-$ be the set of nondecreasing functions $(-\infty ,0]\rightarrow [0,1]$. Let
$$\Omega_{\phi,\psi}^+:=\left\{ \begin{array}{cl}
\widetilde\omega +d\phi(s)\wedge dt & \mbox{on }~ X_+^0\cap X_+^{01};\\
\Omega_1^+ & \mbox{on }~ X_+^1\cap X_+^{01};\\
d(\psi(s)\lambda_-) & \mbox{on }~ X_+^2,
\end{array} \right.$$
where $(\phi,\psi)\in \mathcal{C}_+\times \mathcal{C}_-$.\footnote{$\phi,\psi$ used here are not to be confused with $\phi,\psi$ which appeared in Section~\ref{subsection: W+}.} Then the {\em energy} of a $X_+$-curve $u:F\to X_+$ from $[{\bf y} , i]$ to $\gamma$ is given by:
\begin{equation}
E(u)=\sup_{\phi,\psi} \int_{F} u^*\Omega_{\phi,\psi}^+,
\end{equation}
where the supremum is taken over all pairs $(\phi,\psi)\in \mathcal{C}_+\times \mathcal{C}_-$ such that $\Omega_{\phi,\psi}^+$ is smooth.\footnote{Recall that $\Omega_1^+$ is given by Equation~\eqref{primitive 2} and agrees with $\Omega_0^+$, given by Equation~\eqref{primitive 1}, on their overlap. This imposes conditions on $(\phi,\psi)$.}
\end{defn}

The condition imposed on the intersection with $S_{(z')^f}$ gives an energy bound:

\begin{lemma}[Energy bound]\label{lemma: energy bound}
For all $k\in \N$, there exists $N_k >0$ such that $E(u) \leq N_k$ for all ${\bf y}\in \mathcal{S}_{\bs\alpha,\bs\beta}$, $\gamma\in \mathcal{O}$, and $u\in \mathcal{M}^{\mathcal{F}= k}_{J^+} ({\bf y}, \gamma)$.
\end{lemma}

\begin{proof}
Let $u:(F,j)\to (X_+,J^+)$ be a holomorphic map in $\mathcal{M}^{\mathcal{F}=k}_{J^+} ({\bf y}, \gamma)$.  By Lemma~\ref{properties}(2),(3), $\Omega_{X_+}=d\Theta^+$ on $X_+^\circ:=X_+-N(S_{z^f})$ and the Lagrangian $L_{\bs\alpha}$ is $\Theta^+$-exact. Hence $\int_{\partial F} u^* \Theta^+$ only depends on ${\bf y}$.

Let $v: F'\to X_+^\circ$ be a representative of the homology class $[u]-k[\Sigma]\in H_2 (\check{X}_+ ,Z_{{\bf y} ,\gamma} )$. Since the energy is obtained by integrating a closed form,
$$E(u) = E(v) +k\int_{\Sigma} \widetilde\omega.$$

Now $\Omega_{X_+}^{\phi,\psi}=d\Theta^+_{\phi,\psi}$ on $X_+^\circ$, where
$$\Theta^+_{\phi,\psi}=\left\{ \begin{array}{cl}
\lambda_{+,s}+ \phi(s)dt & \mbox{on }~ X_+^0\cap X_+^\circ\cap X_+^{01};\\
\Theta_1^+ & \mbox{on }~ X_+^1\cap X_+^\circ\cap X_+^{01};\\
\psi (s)\lambda_- & \mbox{on }~ X_+^2.
\end{array} \right.$$
By Equations~\eqref{primitive 1} and \eqref{primitive 2}, $\Theta_1^+$ can be written as $\lambda_{+,s}+ (s+{\pi\over 10})dt$ on $X_+^0\cap X_+^1\cap X_+^\circ\cap X_+^{01}$. Since ${\pi\over 10}<1$, there exist smooth $\Theta^+_{\phi,\psi}$ that extend $\Theta_1^+$, i.e., there exist $\phi\in \mathcal{C}_+$ that extend $\phi(s)=s+{\pi\over 10}$ near $s=0$.

By Stokes' theorem,
\begin{align} \label{equation: E of v}
E(v)& \leq  \int_{\{s\}\times [0,1]\times{\bf y}}\lambda_++\sup_{\phi\in \mathcal{C}_+}\lim_{s\to\infty} \int_{\{s\}\times [0,1]\times{\bf y}}\phi dt \\
\notag & \quad + \int_{\partial F} v^* \lambda_{+}+ \sup_{\phi \in \mathcal{C}_+} \int_{\partial F} \phi  dt -\inf_{\psi\in\mathcal{C}_-}\int_{\gamma} \psi \lambda_-\\
\notag & \leq 4g+ \int_{[0,1]\times{\bf y}} \lambda_+ +\int_{\partial F} v^* \lambda_{+}.
\end{align}
Recall that $\lambda_{+,s}=\lambda_+$ for $s\geq{3\over 2}$. In the above calculation,
$$\sup_{\phi\in \mathcal{C}_+}\lim_{s\to\infty} \int_{\{s\}\times [0,1]\times{\bf y}}\phi dt =2g,~~~ \sup_{\phi \in \mathcal{C}_+} \int_{\partial F} \phi  dt= 2g,~~~ \inf_{\psi\in\mathcal{C}_-}\int_{\gamma} \psi \lambda_-=0.$$
We then obtain
$$E(u) \leq  4g+\int_{[0,1]\times{\bf y}} \lambda_+ + \int_{\partial F} u^* \lambda_{+} + k\int_{\Sigma} \widetilde\omega,$$
which is the desired bound.
\end{proof}

\subsection{Regularity}

Define the subset $\mathcal{M}^h_{J^+} ({\bf y},\gamma,A)\subset \mathcal{M}_{J^+} ({\bf y},\gamma,A)$ consisting of holomorphic curves without vertical fiber components.  As in Lemma~I.\ref{P1-lemma: HF regularity for overline J}, the set $\mathcal{J}_{X_+}^{reg}$ of {\it regular} $J^+ \in \mathcal{J}_{X_+}$ for which all the moduli spaces $\mathcal{M}^h_{J^+} ({\bf y},\gamma,A)$ are transversally cut out is a dense subset of $\mathcal{J}_{X_+}$. We can restrict attention to $\mathcal{M}^h_{J^+} ({\bf y},\gamma,A)$ for the following reason:

\begin{lemma}
If $J^+\in \mathcal{J}_{X_+}^{reg}$ and $u\in \mathcal{M}_{J^+} ({\bf y},\gamma,A)-\mathcal{M}^h_{J^+} ({\bf y},\gamma,A)$, then $I_{X_+}(u)\geq 2+2g$.
\end{lemma}

\begin{proof}
Suppose $u=u_1\cup u_2$, where $u_1$ is regular and $u_2$ is homologous to $k\geq 1$ times a fiber. Since $\langle u_1,u_2\rangle=k\cdot 2g$,
\begin{align*}
I(u) &= I(u_1)+I(u_2)+2k(2g)\\
& \geq 0+ k(2-2g) +4kg \geq k(2+2g).
\end{align*}
Here $I(u_1)\geq 0$ since $I(u_1)\geq \op{ind}(u_1)$ by the index inequality and $\op{ind}(u_1)\geq 0$ by the regularity of $u_1$.
\end{proof}

\subsection{Holomorphic curves in $X_+$ without positive ends}

In this subsection and the next, we make essential use of the assumption $g(S)\geq 2$.

Let $S''= S_{1/2}-A_{[0,N]}$ and $\overline{S}''=S''\cup \{\infty\}$. We define the ``projection'' $\pi_{\overline{S}''}: X_+\to \overline{S}''$ as follows:
\begin{itemize}
\item on $X_+^0$, $\pi_{\overline{S}''}(s,x,t)=x$ if $x\in S''$ and $\pi_{\overline{S}''}(s,x,t)=\infty$ if $x\not\in S''$;
\item on $X_+^1$, $\pi_{\overline{S}''}(x,r_2,\theta_2)=x$ if $x\in S''$ and $\pi_{\overline{S}''}(x,r_2,\theta_2)=\infty$ if $x\not\in S''$;
\item $\pi_{\overline{S}''}(X_+^2)=\infty$.
\end{itemize}

\begin{lemma} \label{genus greater than 2}
If $u: \dot F\to (X_+,J^+)$ is a holomorphic map without positive ends, then $g(F)\geq 2$.
\end{lemma}

\begin{proof}
The map $\pi_{\overline{S}''}\circ u$ can be extended to a continuous map $f: F\to \overline{S}''$. Observe that the curve $u$ must intersect $S_{(z')^f}$ because the symplectic form is exact on $X_+-S_{(z')^f}$. Hence $\deg f>0$. Now we use the following fact: If $f:\Sigma_1\to \Sigma_2$ is a positive degree map between closed oriented surfaces, then $g(\Sigma_1)\geq g(\Sigma_2)$. Since $g(S)=g(\overline{S}'')\geq 2$, it follows that $g(F)\geq 2$.
\end{proof}

\begin{lemma} \label{closed curves}
There are no $I=0$ closed holomorphic curves in $(X_+,J^+)$.
\end{lemma}

\begin{proof}
We argue by contradiction. Let $A=[u_*(F)]$.  By Lemma~\ref{flu}, the intersection form on $H_2(X_+)$ is trivial.  Hence $A\cdot A=0$. If $I(A)=A\cdot A+ c_1(A)=0$,  then it follows that $c_1(A)=0$.

Suppose that $u$ is simple. Then $\chi(F)\geq 0$ by the adjunction formula. This contradicts Lemma~\ref{genus greater than 2}. In particular $I(u)>0$ by the regularity of $u$ and the index inequality. If $v$ is a degree $d$ branched cover of $u$ in the class $A$, then $I(v)=I(dA)=dI(A)\geq d$ using the formula
\begin{equation}
I(dA)= dI(A) +(d^2-d) A\cdot A.
\end{equation}
\end{proof}

\begin{lemma} \label{curves with negative end}
A multiply-covered holomorphic curve $u$ with only negative ends has $I(u)>0$.
\end{lemma}

\begin{proof}
This follows from Lemma~\ref{flu} and the inequality 
\begin{equation}
I(dC)\geq dI(C) + {(d^2-d)\over 2}(2g(C)-2-\op{ind}(C)+h)
\end{equation}
from \cite[Section 5.1]{Hu2}, where $C$ is a simple curve, $\op{ind}(C)$ is the Fredholm index of $C$, and $h$ is the number of hyperbolic ends.
\end{proof}

\subsection{The map $\Phi^+$}

Let $J^+ \in \mathcal{J}_{X_+}^{reg}$. The chain map $\Phi^+$ is given as follows:
$$\Phi^+: (CF^+ (\Sigma,\bs\alpha,\bs\beta,z^f),\bdry)\rightarrow (ECC(M,\lambda_-),\bdry'),$$
$$[{\bf y},i]\mapsto\sum_{\gamma,A} \# \mathcal{M}^{\mathcal{F}=i,I_{X_+}=0}_{J^+} ({\bf y},\gamma,A) \cdot  \gamma,$$
where the summation is over all $\gamma \in \mathcal{O}_{\lambda_-}$ and $A \in H_2 (\check{X}_+,Z_{{\bf y},\gamma})$.  Here $\bdry'$ is the usual ECH differential on $ECC(M,\lambda_-)$.

By a combination of Lemma~\ref{lemma: energy bound} and the Gromov-Taubes compactness theorem (cf.\ Section~I.\ref{P1-subsection: compactness PFH case}), the sum in the definition of $\Phi^+$ is finite. Hence $\Phi^+$ is well-defined.

\begin{thm}\label{thm: chain map}
If $g(S)\geq 2$, then $\Phi^+$ is a chain map.
\end{thm}

\begin{proof}
Similar to that of Theorem~I.\ref{P1-thm: Phi is a chain map}, with slight modifications in view of Lemmas~\ref{closed curves} and \ref{curves with negative end}.
\end{proof}

\begin{rmk}
One can define the twisted coefficient analog of $\Phi^+$, taking into account Lemma~\ref{flu}.
\end{rmk}

\subsection{Restriction to $\Phi$}

In this subsection $\delta$ still denotes the constant that appears in the construction of $\lambda_-$. Let $\mathcal{P}|_N$ be the subset of $\mathcal{P}$ consisting of orbits that are contained in $N$. Also let $\gamma_\theta\in \mathcal{P}_-$ be the orbit corresponding to $\theta\in \bdry S_0$.

\begin{lemma} \label{lemma1}
For $\delta>0$ sufficiently small, if $u\in \mathcal{M}^{\mathcal{F}=0}_{J^+} ({\bf y},\gamma)$, ${\bf y}\in \mathcal{S}_{\bs\alpha,\bs\beta}$, ${\bf y}\subset S_0$, and $\gamma \in \mathcal{O}$, then $\gamma$ is constructed from $\mathcal{P}|_N\cup \{ e',h'\}$.
\end{lemma}

\begin{proof}
This is similar to Lemma~\ref{lemma: energy bound}. Let $u : (F,j) \rightarrow (X_+,J^+)$ be a holomorphic map in $\mathcal{M}^{\mathcal{F}=0}_{J^+} ({\bf y},\gamma)$.  Since $\op{Im}(u)\cap S_{(z')^f}=\varnothing$, we may assume that $\op{Im}(u)\cap N(S_{z^f})=\varnothing$ by possibly modifying $u$ in its homology class. By considerations similar to those of Lemma~\ref{lemma: energy bound}, we obtain:
$$4g+\int_{[0,1]\times{\bf y}} \lambda_+ +\int_{\partial F} u^*\lambda_+ \geq \mathcal{A}_{\lambda_-}(\gamma),$$
where $\mathcal{A}_{\lambda_-}(\gamma)$ is the action of $\gamma$ with respect to $\lambda_-$.
Since the Lagrangian $L_{\bs\alpha}$ is exact by Lemma~\ref{properties}(3), we may take $u(\bdry F)\subset \bdry W_+$. Hence there exist upper bounds for $\int_{\partial F} u^*\lambda_+$ and $\int_{[0,1]\times{\bf y}} \lambda_+$ which are independent of ${\bf y}$ and $\delta$. By Lemma~\ref{properties}(7), all the orbit sets $\gamma$ in $int(N(K))\cup \mathcal{N}$ satisfy $\mathcal{A}_{\lambda_-}(\gamma)\geq {1\over 2\delta}-\kappa$. Hence, for $\delta>0$ sufficiently small, no negative end of $u$ is asymptotic to an orbit in $int(N(K))\cup \mathcal{N}$.
\end{proof}

\begin{lemma} \label{lemma2}
If $u\in \mathcal{M}^{\mathcal{F}=0}_{J^+} ({\bf y},\gamma)$, where ${\bf y}\in \mathcal{S}_{\bs\alpha,\bs\beta}$, ${\bf y}\subset S_0$, and $\gamma$ is constructed from $\mathcal{P}|_N \cup \{ e',h'\}$, then $\op{Im}(u)\subset W_+$ and $\gamma \in \mathcal{O}|_N$.
\end{lemma}

\begin{proof}
Let $u\in \mathcal{M}^{\mathcal{F}=0}_{J^+} ({\bf y},\gamma)$ such that $u(F)\not\subset W_+$.

Suppose that $u$ is not a multi-level Morse-Bott building. Then $u(F) \cap C_{\theta_0} \neq \varnothing$ for some $\theta_0 \in \partial S_0 -\bs\alpha -\bs\beta$, and moreover we may assume that $\gamma_{\theta_0}$ is not an asymptotic limit of $u$ at $-\infty$. Since $J^+$ is admissible, all the curves $C_\theta$ are holomorphic. Hence $\langle u(F),C_{\theta_0}\rangle> 0$ by the positivity of intersections.

Let $D_{\theta}$, $\theta\in \bdry S_0$, be a meridian disk of the solid torus $\mathcal{N}\cup N(K)$ that is bounded by $\{ \theta \} \times \R /2\Z$ and is disjoint from $\{ e',h'\}$, and let $D_{\theta , s} = \{ s\} \times D_{\theta} \subset X_+^2$, where $s<0$ and $\theta\in \bdry S_0$. We then define $$C_{\theta ,s_0} := (C_\theta -\{ s<s_0\} )\cup D_{\theta ,s_0 },$$
where $s_0<0$. When $s_0$ is sufficiently negative, the curve $u(F)$ intersects $C_{\theta_0 ,s_0 }$ only in the region $C_{\theta_0} -\{ s<s_0 \}$, since $\gamma$ is constructed from $\mathcal{P}|_N \cup \{ e',h'\}$ and $D_{\theta_0}$ does not intersect $e'$ and $h'$. Hence $\langle u(F),C_{\theta_0 ,s_0}\rangle >0$. Now, since $[S_{(z')^f}]=[C_{\theta_0 ,s_0}]$ in $H_2 (\check{X}_+, \partial \check{X}_+ -Z_{{\bf y},\gamma} )$, we have
$$\mathcal{F}(u)=\langle [u],S_{(z')^f}\rangle=\langle [u],C_{\theta_0 ,s_0}\rangle>0.$$
This contradicts our assumption that $\mathcal{F}(u)=0$.

If $u$ is a multi-level Morse-Bott building, then we need to make the appropriate modifications (left to the reader), but the same argument goes through. For example, we need to replace $C_{\theta_0}$ by a multi-level building $C_{\theta_0}\cup (\R\times\gamma_{\theta_0})\cup\dots\cup (\R\times\gamma_{\theta_0})$.  Note that if $u$ is a Morse-Bott building, then it could have a component $u_1$ with a negative end that limits to some $\gamma_{\theta_1}$, followed by a gradient trajectory from $\theta_1$ to $\theta_2$, and then by a component $u_2$ with a positive end that limits to $\gamma_{\theta_2}$.
\end{proof}

\begin{thm}\label{thm: control}
For $\delta>0$ sufficiently small, if $u\in \mathcal{M}^{\mathcal{F}=0}_{J^+} ({\bf y},\gamma)$, ${\bf y}\in \mathcal{S}_{\bs\alpha,\bs\beta}$, ${\bf y}\subset S_0$, and $\gamma \in \mathcal{O}$, then $\op{Im}(u)\subset W_+$ and $\gamma\in\mathcal{O}|_N$.
\end{thm}

\begin{proof}
Follows from Lemmas~\ref{lemma1} and \ref{lemma2}.
\end{proof}

The restriction $\Phi$ of $\Phi^+$ to $(W_+,J_+)$ is given as follows:
$$\Phi: \widehat{CF}(\Sigma,\bs\alpha,\bs\beta,z^f)\to ECC_{2g}(M,\lambda_-),$$
$$[{\bf y},0]\mapsto \sum_{\gamma,A} \#\mathcal{M}_{J_+}^{I_{W_+}=0} ({\bf y},\gamma,A)\cdot\gamma,$$
where $\mathcal{M}_{J_+}^{I_{W_+}=0}({\bf y},\gamma,A)$ is the subset of $\mathcal{M}_{J^+}({\bf y},\gamma,A)$ consisting of curves with image in $W_+$.

\begin{thm} \label{thm: Phi is a quasi-isom}
$\Phi$ is a quasi-isomorphism.
\end{thm}

\begin{proof}
The almost complex structure $J_+$ is sufficiently close to $J_+^0$.  For $J_+^0$, the analogous chain map was shown to be a quasi-isomorphism (Theorem~II.\ref{P2-thm: isomorphism}). Considerations similar to those of Theorem~I.\ref{P1-thm: equiv of ECH and PFH} imply that $\Phi$ is a quasi-isomor\-phism.
\end{proof}

\subsection{Commutativity with the $U$-map}

Let $z^b$ be a point in $\R \times [0,1]$ with $t$-coordinate ${1\over 2}$ and let $z=(z^b, z^f ) \in X$. Let $U_z$ be the geometric $U$-map with respect to $z$ on the HF side. On the ECH side, let $z'= (s,z^M)$ be a generic point in $\R \times int(N(K))$ near the binding $K$. We define $U'=U'_{z'}$ so that $\langle U'(\gamma),\gamma'\rangle$ is the count of $I_{ECH} =2$ curves in the symplectization $(\R \times M,J')$ from $\gamma$ to $\gamma'$ that pass through $z'$.

\begin{thm}\label{thm: commutativity}
There exists a chain homotopy
$$K : CF^+ (\Sigma,\bs\alpha,\bs\beta,z^f) \rightarrow ECC (M,\lambda_-)$$ which satisfies
$$U' \circ \Phi^+ -\Phi^+ \circ U_z =\partial' \circ K +K\circ \partial.$$
\end{thm}

\begin{proof}
The commutativity of $\Phi^+$ with the $U$-maps up to homotopy is obtained by moving the point constraint in the cobordism $X_+$ from $s=+\infty$ to $s=-\infty$.

The $1$-parameter family of points $(z_\tau)_{\tau \in \R}$ is chosen as follows:  For $\tau\geq 0$, let $z_\tau =(z_\tau^b,z^f)$, where $z^b_\tau$ approaches $(s,t)=(+\infty,{1\over 2})$ as $\tau \rightarrow +\infty$ and $z_0^b$ is near the center of the disk $D^2=\{r_2\leq 1\}$. Next, for $\tau \in [-1,0]$, let $z_\tau=(z_0^b,z^f_\tau)$ so that $(z_0^b,z_{-1}^f)\in \{0\}\times \widetilde{\mathcal{B}}$ is near the binding $K$. For $\tau\leq -1$, let $z_\tau=(\tau+1,z^M) \in (-\infty ,0] \times M$, where $z^M   \in M=\widetilde{\mathcal{B}}$ is a point near the binding. Finally, we consider a small perturbation of $(z_\tau)_{\tau \in \R}$ to make it generic (without changing its name).

We define the $1$-parameter family of almost complex structures $(J_\tau^+)_{\tau\in \R}$ so that $J_\tau^+$ is $C^l$-close to $J^+$ and agrees with $J^+$ outside a small neighborhood of $z_\tau$.

The rest of the chain homotopy argument is standard, with the exception of the obstruction theory that was carried out in \cite{HT1,HT2}.
\end{proof}

\begin{thm}\label{thm: K}
For $\delta>0$ sufficiently small, if ${\bf y}\in \mathcal{S}_{\bs\alpha,\bs\beta}$ and ${\bf y}\subset S_0$, then $K([{\bf y} ,0])=0$.
\end{thm}

\begin{proof}
Similar to the proof of Theorem~\ref{thm: control}.
The coefficient $\langle K([{\bf y} ,0]), \gamma \rangle$ is given by the count of $I_{X_+}=1$ curves from ${\bf y}$ to $\gamma$ that pass through $z_\tau$ for some $\tau$ and do not intersect $S_{(z')^f}$. If such a curve $u$ exists, then $\op{Im}(u) \not\subset W_+$.  This is not possible by Theorem~\ref{thm: control}.
\end{proof}

\section{Proof of Theorem~\ref{thm: HF+=ECH}}\label{section: proof}

In this section we prove Theorem~\ref{thm: HF+=ECH}. In Section~\ref{subsection: algebraic} we prove an algebraic result (Theorem~\ref{thm: algebraic}) which is sufficient to prove that $\Phi^+$ is a quasi-isomorphism.  The conditions of Theorem~\ref{thm: algebraic} are verified in Section~\ref{completion}.

\subsection{Some algebra}\label{subsection: algebraic}

\begin{defn}
Let $(A, d)$ be a chain complex. We say that a chain map $f \colon A \to A$ is
{\em homologically almost nilpotent} (abbreviated {\em han}) if for every $x \in H(A)$
there exists $n \in \N$ such that $(f_*)^n(x)=0$.
\end{defn}

Prototypical examples of {\em han} maps are the $U$-maps in $HF^+$ and $ECH$.

Let $(A, d_A)$ and $(B, d_B)$ be chain complexes with {\em han} maps
$U_A \colon A \to A$ and $U_B \colon B \to B$ and let $\Phi^+ \colon A \to B$ be a chain map such that the diagram
$$\xymatrix{
A \ar[r]^{\Phi^+} \ar[d]_{U_A} & B \ar[d]^{U_B} \\
A \ar[r]^{\Phi^+} & B
}$$
commutes up to a chain homotopy $K$.  We form a chain complex $D=A \oplus A \oplus B \oplus B$ with differential
$$d_D = \left ( \begin{matrix}
d_A &  0    & 0     & 0 \\
U_A & d_A & 0     & 0 \\
\Phi^+  & 0     & d_B & 0 \\
K      & \Phi^+ & U_B & d_B
\end{matrix} \right ).$$

Given a chain map $f$, we denote its mapping cone by $C(f)$.

\begin{lemma}
There is an exact triangle:
\begin{equation} \label{eqn: first triangle}
\xymatrix{
H(C(U_A)) \ar[rr]^{(\Phi_{alg} )_*} & &  H(C(U_B)) \ar[dl] \\
& H(D) \ar[ul] &
}
\end{equation}
where $\Phi_{alg} = \left ( \begin{matrix} \Phi^+ & 0 \\ K & \Phi^+ \end{matrix} \right )$.
\end{lemma}

\begin{proof}
From the shape of $d_D$, it is evident that $(D, d_D)$ is the mapping cone of
$\Phi_{alg} \colon C(U_A) \to C(U_B)$.
\end{proof}

\begin{lemma}
There is an exact triangle:
\begin{equation} \label{eqn: second triangle}
\xymatrix{
H(C(\Phi^+)) \ar[rr]^{{(U_{\Phi^+})}_*} & &  H(C(\Phi^+)) \ar[dl] \\
& H(D) \ar[ul] &
}
\end{equation}
where $U_{\Phi^+}= \left ( \begin{matrix} U_A & 0 \\ K & U_B \end{matrix} \right )$.
\end{lemma}

\begin{proof}
Let $C(\Phi^+)=A\oplus B$ be the cone of $\Phi^+$ with differential $d_{\Phi^+}=\begin{pmatrix} d_A & 0 \\ \Phi^+ & d_B \end{pmatrix}$. Then
$U_{\Phi^+}: (C(\Phi^+),d_{\Phi^+})\to (C(\Phi^+),d_{\Phi^+})$ is a chain map. Hence the complex $(D', d_{D'})$, where $D'= A \oplus B \oplus A \oplus B$ and
$$d_{D'} = \begin{pmatrix} d_{\Phi^+} & 0 \\ U_{\Phi^+} & d_{\Phi^+} \end{pmatrix}= \begin{pmatrix}
d_A  &  0    &  0  &  0  \\
\Phi^+  &  d_B &  0  &  0  \\
U_A &  0     &  d_A & 0 \\
K     & U_B  &  \Phi^+  & d_B
\end{pmatrix}, $$
is the cone of $U_{\Phi^+}$. Moreover $f \colon D \to D'$
where
$$f = \left ( \begin{matrix}
1 & 0 & 0 & 0 \\
0 & 0 & 1 & 0 \\
0 & 1 & 0 & 0 \\
0 & 0 & 0 & 1
\end{matrix} \right ) $$
is an isomorphism of complexes.
\end{proof}

\begin{lemma}
$U_{\Phi^+}$ is a {\em han} map.
\end{lemma}

\begin{proof}
Consider the following commutative diagram with exact rows:
$$\xymatrix{
H(B) \ar[r]^-{i_*} \ar[d]_{U_B^n} & H(C(\Phi^+ )) \ar[r]^-{j_*}  \ar[d]_{U_{\Phi^+}^n} &  H(A)
\ar[d]_{U_A^n} \\
H(B) \ar[r]^-{i_*} \ar[d]_{U_B^m} & H(C(\Phi^+ )) \ar[r]^-{j_*}  \ar[d]_{U_{\Phi^+}^m} &  H(A)
\ar[d]_{U_A^m} \\
H(B) \ar[r]^-{i_*}  & H(C(\Phi^+ )) \ar[r]^-{j_*}   &  H(A)
}$$
Given $x \in H(C(\Phi^+ ))$, we choose $n\in \N$ sufficiently large so that $U_A^n(j_*(x))=j_*(U^n_{\Phi^+}(x))=0$. Then
$U_{\Phi^+}^n(x) = i_*(y)$ for some $y \in H(B)$. Next choose $m\in \N$ sufficiently large so that $U_B^m(y)=0$. Then
$U_{\Phi^+}^{n+m}(x)= U^m_{\Phi^+}(i_*(y))=i_*(U^m_B(y)) =0$.
\end{proof}

\begin{thm}\label{thm: algebraic}
If $\Phi_{alg}$ is a quasi-isomorphism, then $\Phi^+$ is a quasi-isomorphism.
\end{thm}

\begin{proof}
If $\Phi_{alg}$ is a quasi-isomorphism, then $H(D)=0$ by Exact Triangle~\eqref{eqn: first triangle}. This in turn implies that $U_{\Phi^+}$ is a quasi-isomorphism by Exact Triangle~\eqref{eqn: second triangle}. However the {\em han} map $U_{\Phi^+}$ cannot be a quasi-isomorphism, unless $H(C(\Phi^+ ))=0$. Finally, the triangle
$$\xymatrix{
H(A) \ar[rr]^{\Phi^+_*} & & H(B) \ar[dl] \\
& H(C(\Phi^+ )) \ar[ul] &
}$$
implies that $\Phi^+$ is a quasi-isomorphism.
\end{proof}

We finish this subsection with a lemma which compares the homology of $C(U)$  with that of $\ker U$.

\begin{lemma} \label{lemma: reduction to kernel}
Let $(C, d)$ be a chain complex and let $U:C \to C$ be a chain map. If $U$ is surjective, then the inclusion
\begin{align*}
i : \ker U & \to C(U) \\
x & \mapsto \binom{x}{0}
\end{align*}
is a quasi-isomorphism.
\end{lemma}

\begin{proof}
Let $\overline{U} : C/\ker U \to C$ be the map induced by $U$. We have a short exact sequence of complexes
$$ 0 \to \ker U \to C(U) \to C(\overline{U}) \to 0,$$
which induces the exact triangle:
$$ \xymatrix{
H(\ker U) \ar[rr]^{i_*} & & H(C(U)) \ar[dl] \\
& H(C(\overline{U})) \ar[ul] &
}$$
Since $U$ is surjective, $\overline{U}$ is an isomorphism. Hence $H(C(\overline{U}))=0$ and the lemma follows.
\end{proof}

\subsection{Heegaard Floer chain complexes} \label{subsection: HF chain complexes}

Recall the subcomplex $\widehat{CF'}(S_0,{\bf a},\hh({\bf a}))$ of $\widehat{CF}(\Sigma,\bs\alpha,\bs\beta,z^f)$ from Section~I.\ref{P1-subsubsection: variant CF of S}, which is generated by $\mathcal{S}_{{\bf a}, \hh({\bf a})}$; let
$$j': \widehat{CF'}(S_0,{\bf a},\hh({\bf a}))\to \widehat{CF}(\Sigma,\bs\alpha,\bs\beta,z^f)$$ be the natural inclusion map. We are viewing
$$\widehat{CF}(\Sigma,\bs\alpha,\bs\beta,z^f)\subset CF^+(\Sigma,\bs\alpha,\bs\beta,z^f)$$
as the subcomplex generated by elements of the form $[{\bf y},0]$. The chain complex $\widehat{CF}(S_0,{\bf a},\hh({\bf a}))$ is the quotient $\widehat{CF'}(S_0,{\bf a},\hh({\bf a}))/\sim$, defined in Section~I.\ref{P1-subsubsection: variant CF of S}.

\begin{lemma} \label{isom on homology}
There is an isomorphism $j: \widehat{HF}(S_0,{\bf a},\hh({\bf a}))\to \widehat{HF}(\Sigma,\bs\alpha,\bs\beta,z^f)$ given by $[Z]\mapsto [Z].$
\end{lemma}

\begin{proof}
This follows from the discussion of Theorem~I.\ref{P1-t:hf}.  Note that the natural candidate
$$\widehat{CF}(S_0,{\bf a},\hh({\bf a}))\to \widehat{CF}(\Sigma,\bs\alpha,\bs\beta,z^f),\quad [Z]\to Z$$
for a chain map is not a well-defined map.
\end{proof}

\begin{lemma}  \label{1}
The inclusion $i:\widehat{CF}(\Sigma,\bs\alpha,\bs\beta,z^f)\to C(U)$ given by ${\bf y} \mapsto \begin{pmatrix} [{\bf y},0] \\ 0 \end{pmatrix}$ is a quasi-isomorphism.
\end{lemma}

\begin{proof}
This follows from Lemma~\ref{lemma: reduction to kernel}, since $U([{\bf y},i])=[{\bf y},i-1]$ for $i\geq 1$ and $\ker U\simeq \widehat{CF}(\Sigma,\bs\alpha,\bs\beta,z^f)$.
\end{proof}

\subsection{ECH chain complexes} \label{subsection: ECH chain complexes}

We describe several ECH chain complexes that are related to $(ECC(M,\lambda_-),\bdry')$ and are constructed from certain subsets $\mathcal{S}$ of the set $\mathcal{P}=\mathcal{P}_{\lambda_-}$ of simple orbits of $R_{\lambda_-}$.  Many of these appeared in \cite[Section~\ref{P0-section: proof of theorem equivalence of ECHs}]{CGH1}. Let $U'$ be the $U$-map of $ECC(M,\lambda_-)$ with respect to $(s_0,z^M)\in \R\times M$, where $z^M$ is a generic point which is sufficiently close to the binding.

Let $\mathcal{O}_{\mathcal{S}}$ be the set of orbit sets that are constructed from $\mathcal{S}$. Then $\mathcal{S}$ is {\em closed} if $\gamma'\in \mathcal{O}_{\mathcal{S}}$, whenever $\gamma\in \mathcal{O}_{\mathcal{S}}$, $\gamma'\in \mathcal{O}_{\mathcal{P}}$, and $\langle \bdry'\gamma,\gamma'\rangle\not=0$ or $\langle U'\gamma,\gamma'\rangle\not=0$. If $\mathcal{S}$ is closed, then let $(A_{\mathcal{S}},\bdry'_{\mathcal{S}})$ be the subcomplex of $ECC(M,\lambda_-)$ generated by $\mathcal{O}_{\mathcal{S}}$ and let $U'_{\mathcal{S}}$ be the restriction of $U'$ to $A_{\mathcal{S}}$. Let $\mathcal{P}|_N\subset \mathcal{P}$ be the set of orbits in the mapping torus $N$. The subsets
$$\mathcal{S}_1=\mathcal{P}|_N\cup\{e', h'\},~ \mathcal{S}_2=\mathcal{P}|_{N\cup\mathcal{N}}\cup\{e',h'\}, ~ \mathcal{P}|_N\cup\{h'\},~ \mathcal{P}|_{N\cup\mathcal{N}}\cup\{h'\},~\mathcal{P}|_N$$ are closed and we write $A_i=A_{\mathcal{S}_i}$, $\bdry'_i=\bdry'_{\mathcal{S}_i}$, and $U_i'=U'_{\mathcal{S}_i}$ for $i=1,2$, as well as
$$\widehat{ECC}^\natural(N)=A_{\mathcal{P}|_N\cup\{h'\}},~~ \widehat{ECC}^{\natural\natural}(N)=A_{\mathcal{P}|_{N\cup\mathcal{N}}\cup\{h'\}},~~ ECC(N)=A_{\mathcal{P}|_N}.$$
Also let $ECC_{2g}(N)\subset ECC(N)$ be the subcomplex generated by orbit sets $\gamma$ satisfying $\langle \gamma,S\times\{t\}\rangle =2g$. Let
$$q_1: ECC_{2g}(N)\to \widehat{ECC}^{\natural}(N), \quad q_2: ECC_{2g}(N)\to \widehat{ECC}^{\natural\natural}(N)$$ be the chain maps given by the natural inclusion. Then we have the following:

\begin{lemma} \label{q}
The chain maps $q_1$ and $q_2$ are quasi-isomorphisms.
\end{lemma}

\begin{proof}
The chain map $q_1$ is a quasi-isomorphism by Section~II.\ref{P2-section: stabilization} and \cite[{Section~\ref{P0-subsec: U map}}]{CGH1}. By a direct limit argument similar to that of
 \cite[Proposition~\ref{P0-prop: ECH of two contact forms in M}]{CGH1}, there is a quasi-isomorphism $r: \widehat{ECC}^{\natural\natural}(N)\to \widehat{ECC}^{\natural}(N)$ such that $r\circ q_2 = q_1$. This implies that $q_2$ is also a quasi-isomorphism.
\end{proof}

\begin{lemma} \label{2}
The inclusions $p_1:\widehat{ECC}^\natural(N)\to C(U'_1)$ and $p_2: \widehat{ECC}^{\natural\natural}(N)\to C(U'_2)$ given by $\Gamma\mapsto \begin{pmatrix} \Gamma \\ 0 \end{pmatrix}$ are quasi-isomorphisms.
\end{lemma}

\begin{proof}
This follows from Lemma~\ref{lemma: reduction to kernel}. The map $U'_i$, $i=1,2$, is given by:
\begin{equation} \label{defn of U''}
U'_i((e')^k(h')^l\Gamma)=(e')^{k-1}(h')^l\Gamma,
\end{equation}
where $\Gamma\in \mathcal{O}|_N$ or $\mathcal{O}|_{N\cup\mathcal{N}}$; see \cite[Claim~\ref{P0-lemma: computation of U_0}]{CGH1} for a similar calculation. Hence $U'_i$ is surjective, $\ker U'_1= \widehat{ECC}^\natural(N)$, and $\ker U'_2=\widehat{ECC}^{\natural\natural}(N)$. This implies the lemma.
\end{proof}

\begin{lemma} \label{lemma: quasi-isom}
The inclusion $\mathfrak{i}: \widehat{ECC}^{\natural\natural}(N)\to C(U')$ given by $\Gamma\mapsto \begin{pmatrix} \Gamma \\ 0 \end{pmatrix}$ is a quasi-isomorphism.
\end{lemma}

\begin{proof}
This is similar to the argument in \cite[Section~\ref{P0-section: proof of theorem equivalence of ECHs}]{CGH1}.

Choose an identification $\eta: H_1(N(K);\Z)\stackrel\sim\to \Z$ such that the homology class of the binding is $1$.  Define the filtration $\mathcal{F}: ECC(M)\to \Z^{\geq 0}$ such that $$\mathcal{F}\left(\sum_i \gamma_i\otimes \Gamma_i\right)= \max_i\eta([\gamma_i]),$$
where $\gamma_i\in \mathcal{O}|_{N(K)}$ and $\Gamma_i\in \mathcal{O}|_{N\cup\mathcal{N}}$. Let $\mathcal{F}^{\natural\natural}: \widehat{ECC}^{\natural\natural}(N)\to \Z^{\geq 0}$ be its restriction to $\widehat{ECC}^{\natural\natural}(N)$. (Note that $\mathcal{F}^{\natural\natural}$ is a trivial filtration.) Next define the filtration $\widehat{\mathcal{F}}: C(U')\to \Z^{\geq 0}$ such that
$$\widehat{\mathcal{F}}\left( \begin{array}{c} \sum_i \gamma_i\otimes\Gamma_i \\ \sum_j \gamma_j'\otimes \Gamma_j'\end{array}  \right)= \max_{i,j} \{ \eta([\gamma_i]),\eta([\gamma'_j])\}.$$

The map ${\frak i}$ is an $(\mathcal{F}^{\natural\natural},\widehat{\mathcal{F}})$-filtered chain map.  The induced map
$$E^1({\frak i}): E^1(\mathcal{F}^{\natural\natural})\to E^1(\widehat{\mathcal{F}})$$
on the $E^1$-level agrees with the isomorphism $(p_2)_*$; the proof is similar to that of \cite[Section~\ref{P0-section: proof of theorem equivalence of ECHs}]{CGH1}. If a filtered chain map between filtered chain complexes which are bounded below is an isomorphism on the $E^r$-level, then it is a quasi-isomorphism. This implies that ${\frak i}$ is a quasi-isomorphism.
\end{proof}

\subsection{Completion of proof of Theorem~\ref{thm: HF+=ECH}} \label{completion}

By Theorems~\ref{thm: U-map}, \ref{thm: chain map}, and \ref{thm: commutativity}, the map
$$\Phi^+: CF^+(\Sigma,\bs\alpha,\bs\beta,z^f)\to ECC(M,\lambda_-)$$
is a chain map which commutes with $U$ and $U'$ up to the chain homotopy $K^+=K+\Phi^+ \circ H$, where $H$ is given in Theorem~\ref{thm: U-map} and $K$ is given in Theorem~\ref{thm: commutativity}. Here $U$ is the formal $U$-map on $(CF^+(\Sigma,\bs\alpha,\bs\beta,z^f),\bdry)$ and $U'$ is the $U$-map on $(ECC(M,\lambda_-),\bdry')$. In view of Theorem~\ref{thm: algebraic}, Theorem~\ref{thm: HF+=ECH} immediately follows from:

\begin{thm} \label{thm: phi alg}
The algebraic map $\Phi_{alg}$ is a quasi-isomorphism.
\end{thm}

Let $\Phi': \widehat{CF'}(S_0,{\bf a},\hh({\bf a}))\to ECC_{2g}(N)$ be the map from Definition~I.\ref{P1-defn: Phi'}. The map $\Phi'$ descends to $\Phi: \widehat{CF}(S_0,{\bf a},\hh({\bf a}))\to ECC_{2g}(N)$, which was shown to be a quasi-isomorphism in \cite{CGH2,CGH3}.  Here we are using $ECC_{2g}(N)$ instead of $PFC_{2g}(N)$, but there is no substantial difference.

Observe that there is a discrepancy between the algebra and the geometry: the map $\Phi_{alg}$ which we are using here is not the map $\Phi$, and we need to reconcile the two.

\begin{proof}
If $Z\in \widehat{CF'}(S_0,{\bf a},\hh({\bf a}))$, then $\Phi^+(Z)=\Phi'(Z)$ by Theorem~\ref{thm: control}.
We observed in Theorem~\ref{thm: U-map} that  $H(Z)=0$. Moreover, $K(Z)=0$ by Theorem~\ref{thm: K} and thus $K^+(Z)=0$. Hence
$$\Phi_{alg} \left ( \begin{matrix} Z \\ 0 \\ \end{matrix} \right ) = \left ( \begin{matrix} \Phi^+ (Z ) \\ K^+(Z) \\ \end{matrix} \right ) =\left ( \begin{matrix} \Phi' (Z ) \\ 0 \\ \end{matrix} \right ),$$
and the following diagram is commutative:
$$\xymatrix{
\widehat{CF'}(S_0,{\bf a},\hh({\bf a})) \ar[r]^-{\Phi'} \ar[d]_{i\circ j'} & ECC_{2g}(N) \ar[d]_{{\frak i}\circ q_2} \\
C(U)  \ar[r]^-{\Phi_{alg}} & C(U').
}$$
This gives rise to the following commutative diagram of homology groups:
$$\xymatrix{
\widehat{HF}(S_0,{\bf a},\hh({\bf a})) \ar[r]^-{\Phi_*} \ar[d]_{i_*\circ j} & ECH_{2g}(N) \ar[d]_{({\frak i}\circ q_2)_*} \\
H(C(U))  \ar[r]^-{(\Phi_{alg})_*} & H(C(U')).
}$$
Since $j$, $i_*$, $\Phi_*$, $(q_2)_*$, and ${\frak i}_*$ are isomorphisms by Lemma~\ref{isom on homology}, Lemma~\ref{1}, \cite{CGH2,CGH3}, Lemma~\ref{q}, and Lemma~\ref{lemma: quasi-isom},  $\Phi_{alg}$ itself is a quasi-iso\-mor\-phism.
\end{proof}

\vskip0.5cm

\n {\em Acknowledgements.} We are indebted to Michael Hutchings for
many helpful conversations and for our previous collaboration which
was a catalyst for the present work. We also thank Tobias Ekholm,
Dusa McDuff, Ivan Smith and Jean-Yves Welschinger for illuminating
exchanges.  Part of this work was done while KH and PG visited MSRI
during the academic year 2009--2010. We are extremely grateful to
MSRI and the organizers of the ``Symplectic and Contact Geometry and
Topology'' and the ``Homology Theories of Knots and Links'' programs
for their hospitality; this work probably would never have seen the
light of day without the large amount of free time which was made
possible by the visit to MSRI.

\end{document}